\documentclass[article]{elsarticle}

\usepackage{color}
\definecolor{olga}{rgb}{0., 0., 0.0}
\definecolor{felipe}{rgb}{0., 0., 0.0}
\definecolor{damiano}{rgb}{0., 0., 0.0}
\definecolor{discarded}{rgb}{0.9098,0.1294,0.2078}

\newcommand{\om}[1]{\textcolor{olga}{#1}}

\usepackage{slashbox}

\usepackage{lipsum}
\usepackage{amsfonts}
\usepackage{graphicx}
\usepackage{epstopdf}
\usepackage{algorithmic}
\ifpdf
  \DeclareGraphicsExtensions{.eps,.pdf,.png,.jpg}
\else
  \DeclareGraphicsExtensions{.eps}
\fi

\usepackage[margin=1.in]{geometry}

\title{Fast reconstruction of 3D blood flows from Doppler ultrasound images and reduced models.}

\author[1]{Felipe Galarce\corref{cor1}} 
\author[1]{Jean-Frédéric Gerbeau}
\author[1]{Damiano Lombardi}
\author[1,2]{Olga Mula}

\cortext[cor1]{felipe.galarce-marin@inria.fr. 2 Rue Simone Iff 75589 Paris Cedex 12. Bureau A316.}
\address[1]{Centre de Recherche INRIA de Paris \& Laboratoire Jacques-Louis Lions, France.}
\address[2]{Paris-Dauphine University, PSL Research University, CNRS, UMR 7534, CEREMADE, France.}
 
\newtheorem{theorem}{Theorem} 
\newtheorem{lemma}[theorem]{Lemma}
\newtheorem{corollary}[theorem]{Corollary}
\newdefinition{remark}{Remark}
\newdefinition{definition}{Definition}
\newproof{proof}{Proof}
\newproof{pot}{Proof of Theorem \ref{thm2}}

\usepackage{amsopn}

\usepackage{mathtools}
\usepackage{subfigure}

\newcommand{\cD}{\ensuremath{\mathcal{D}}}

\newcommand{\cF}{\ensuremath{\mathcal{F}}}

\newcommand{\cH}{\ensuremath{\mathcal{H}}}

\newcommand{\cM}{\ensuremath{\mathcal{M}}}
\newcommand{\cN}{\ensuremath{\mathcal{N}}}
\newcommand{\cO}{\ensuremath{\mathcal{O}}}
\newcommand{\cP}{\ensuremath{\mathcal{P}}}

\newcommand{\cV}{\ensuremath{\mathcal{V}}}
\newcommand{\cW}{\ensuremath{\mathcal{W}}}

\newcommand{\bE}{\ensuremath{\mathbb{E}}}

\newcommand{\bG}{\ensuremath{\mathbb{G}}}

\newcommand{\bP}{\ensuremath{\mathbb{P}}}

\newcommand{\bR}{\ensuremath{\mathbb{R}}}

\newcommand{\rY}{\ensuremath{\mathrm{Y}}}

\def\[{\left[}
\def\]{\right]}
\def\<{\langle}
\def\>{\rangle}
\def\({\left(}
\def\){\right)}
\def\[{\left [}
\def\]{\right]}
\def\({\left(}
\def\){\right)}

\newcommand{\innerp}[2]{\< #1, #2 \>}
\newcommand{\norm}[1]{\Vert #1 \Vert}

\newcommand{\Vn}{\ensuremath{{V_n}}}
\newcommand{\Wm}{\ensuremath{{W_m}}}

\newcommand{\wc}{\ensuremath{\mathrm{wc}}}
\newcommand{\ms}{\ensuremath{\mathrm{ms}}}
\newcommand{\av}{\ensuremath{\mathrm{av}}}

\newcommand{\pbdw}{\ensuremath{\text{(pbdw)}}}
\newcommand{\wcpbdw}{\ensuremath{\text{(wc, pbdw)}}}
\newcommand{\mspbdw}{\ensuremath{\text{(ms, pbdw)}}}
\newcommand{\wcaff}{\ensuremath{\text{(wc, aff)}}}
\newcommand{\msaff}{\ensuremath{\text{(ms, aff)}}}

\newcommand{\wcaffk}{\ensuremath{\text{(wc, aff, k)}}}

\newcommand{\msaffk}{\ensuremath{\text{(ms, aff, k)}}}

\newcommand{\aff}{\ensuremath{\text{(aff)}}}
\newcommand{\affpbdw}{\ensuremath{\text{(aff)}}}
\newcommand{\dd}{\ensuremath{\text{(dd)}}}
\newcommand{\wcdd}{\ensuremath{\text{(wc, dd)}}}

\newcommand{\train}{\ensuremath{\text{train}}}
\newcommand{\test}{\ensuremath{\text{test}}}

\newcommand{\healthy}{\ensuremath{\text{healthy}}}
\newcommand{\sick}{\ensuremath{\text{sick}}}

\newcommand{\ord}{\ensuremath{\mathcal{O}}}

\newcommand{\dist}{\operatorname{dist}}

\DeclareMathOperator*{\argmax}{arg\,max}
\DeclareMathOperator*{\argmin}{arg\,min}

\newcommand{\dt}{\ensuremath{\mathrm dt}}
\DeclareMathOperator*{\vspan}{span}

\newcommand{\eps}{\ensuremath{\varepsilon}}

\newcommand{\pl}{\emph{POD-lin}}
\newcommand{\ppa}{\emph{P-POD-aff}}
\newcommand{\pga}{\emph{P-Greedy-aff}}
\newcommand{\pdba}{\emph{P-DB-aff}}

\usepackage{soul}

\begin{document}

\begin{abstract}
This paper deals with the problem of building fast and reliable 3D reconstruction methods for blood flows for which partial information is given by Doppler ultrasound measurements. This task is of interest in medicine since it could enrich the available information used in the diagnosis of certain diseases which is currently based essentially on the measurements coming from ultrasound devices. The fast reconstruction of the full flow can be performed with state estimation methods that have been introduced in recent years and that involve reduced order models. One simple and efficient strategy is the so-called Parametrized Background Data-Weak approach (PBDW, see \cite{MPPY2015}). It is a linear mapping that consists in a least squares fit between the measurement data and a linear reduced model to which a certain correction term is added. However, in the original approach, the reduced model is built \emph{a priori} and independently of the reconstruction task (typically with a proper orthogonal decomposition or a greedy algorithm). In this paper, we investigate the construction of other reduced spaces which are built to be better adapted to the reconstruction task and which result in mappings that are sometimes nonlinear. We compare the performance of the different algorithms on numerical experiments involving synthetic Doppler measurements. The results illustrate the superiority of the proposed alternatives to the classical linear PBDW approach.
\end{abstract}


\maketitle


\section{Introduction}
\om{Developing artificial intelligence-driven tools to assist doctors in medical decisions and diagnosis requires to solve efficiently and in a reliable manner data assimilation and inverse problems from biomedical applications.} These problems share the following general common features, which will guide our subsequent developments:
\begin{itemize}
\item The available data is often corrupted by noise and obtained with medical imaging techniques, which have the advantage of being \emph{non-invasive}.
\item There may be morphological constraints that prevent from measuring at specific locations.
\item In some cases, the device may not be able to measure directly the desired quantity of interest (QoI), and a complex post-processing may be required to obtain an estimate of it.
\item Sometimes, the desired QoI is a prediction or a forecast ahead in time, which requires a specific treatment to be inferred.
\end{itemize}
\om{In this paper, we develop state estimation techniques which aim at taking the above features into account, and which involve reduced modeling of parametrized Partial Differential Equations (PDEs). As an illustrative example} of how the reconstruction techniques that we deploy address some of the above points, in this paper, we consider a haemodynamics inverse problem: the real-time reconstruction of the full 3D blood velocity field in an artery \om{from} Doppler ultrasound images taken on a restricted portion of the artery. This application is of interest in its own right since a real-time reconstruction of the full flow would enrich the available information for medical diagnosis.

\om{Haemodynamics forward and inverse problems is a wide multidisciplinary topic with a long history, and a complete overview of the numerous existing contributions would go beyond the scope of the present paper. However, as a brief summary of the state of the art for our application on haemodynamics reconstructions using medical images, we could start by citing  \cite{aderson1998}, where a method based on quadratures is used to estimate the velocity from Doppler ultrasound. In \cite{moireau2013,hu2007,caiazzo2017,muller2018,adib2016} sequential data estimation methods are used to provide a description of the haemodynamics in different portions of the vascular tree. A patient specific fluid-structure interaction simulation of the left ventricle is proposed in \cite{lassila2012} based on Magnetic Resonance Imaging data. The calibration of computational fluid mechanics simulations was also investigated in \cite{koltukluouglu2019}. A multi-fidelity approach and a Bayesian framework are detailed in \cite{perdikaris2016}.  Recently, some works have explored kernel and deep learning based approaches to perform data assimilation \cite{koeppl2018,huttunen2020,kissas2020machine}.}  

%

In this work, we use methods based on reduced-order modelling to obtain reconstructions in close to real-time. This approach has attracted considerable attention and numerous mathematical developments and applications have been developed, see \cite{khalil2007linear, astrid2008missing,buffoni2008non, leroux2013application, MM2013, MMPY2015, raiola2015piv, ABGMM2017, KGV2018,CDDFMN2019}. These methods share connections with other well-known approaches such as 3D and 4D-Var (see \cite{Lorenc1981,LT1986}) or the Partial Spline Model (\cite[Chapter 9]{Wahba1990}). We refer to  \cite{Taddei2017, KGV2018} for further details on this point. \om{Regarding inverse haemodynamics problems, prior contributions involving reduced-order modeling are, e.g., \cite{lassila2011,pant2017}.}

Our starting point is the Parametrized Background Data-Weak originally introduced in \cite{MPPY2015}. The method consists in a linear mapping which is built from a least squares fit between the measurements and a given reduced space and an additional term that is capable of correcting model bias to some extent. In the original approach of PBDW, the reduced model is built \emph{a priori} and independently of the reconstruction task (typically with a Proper Orthogonal Decomposition or a greedy algorithm). In this paper, we investigate the construction of other reduced spaces which are built to be better adapted to the reconstruction task and which result in mappings that are sometimes nonlinear. We compare the performance of the different algorithms in the context of the fast reconstruction of the velocity field in an artery. The results illustrate the superiority of the proposed alternatives to the classical approach involving linear reduced bases.
 
The paper is organized as follows, in section \ref{sec:RecMeth} we describe mathematically what we understand by state estimation problems and reconstruction algorithms (section \ref{sec:SE}). We next introduce several ways of building a reduced order space, some of which take the measurements into account, henceforth leading to non-linear reconstruction methods (section \ref{sec:algos}). \om{We give details on the practical implementation and computational costs in section \ref{sec:impl}.} We next apply the methodology to the reconstruction of 3D blood flows from Doppler images (section \ref{sec:numerical_example}). The example is semi-realistic, meant to be a first (still idealized) step towards realistic applications. The main assumptions are the following. First, the Doppler images are noiseless and synthetically computed but they have been generated with a realistic model with respect to the real involved physics and measurement devices. This assumption is mainly due to our lack of real Doppler measurements and also because of the very involved space-time structure of the noise in Doppler images (which is a research topic in itself \cite{ledoux1997,bjaerum2002,demene2015}). Second, we assume that there is no model misfit even if the method can correct it to some extent. However, we use a rather realistic model for flows in large arteries based on incompressible Navier-Stokes equations, with boundary conditions used in haemodynamics simulations (see \cite{Formagia_Cardiovascular_Mathematics}). The precise parametric PDE model upon which we build our reduced models is explained in section \ref{subec:model}. Section \ref{subsec:measurements} gives some details on Doppler ultrasound imaging and the way in which we have incorporated it to our methodology. Finally, sections \ref{subsec:firstTest} and \ref{subsec:secondTest} present numerical results in two test cases. They illustrate the superiority of the proposed reconstruction algorithms with respect to the classical linear PBDW. In addition to this, our second example also shows that the method can be used to estimate quantities of interest which could be helpful in the detection of an arterial blockage.

We conclude this introduction by emphasizing that the assumptions in our computations with no model misfit and synthetic, noiseless measurements are not inherent limitations of the present approach. The focus lies rather on the reduced space construction \om{and the haemodynamics application}. The model misfit is automatically corrected and measurement noise can be added to the recostruction pipeline (see \cite{Taddei2017, ABGMM2017, GMMT2019} for works on measurement noise). We also emphasize the novelty that the current approach represents for the field of blood flow reconstruction. Indeed, estimations performed routinely by medical doctors usually do not involve full reconstructions (perhaps due to the lack of efficient techniques). They often reduce to simple pointwise estimations of the peak and average velocities of the imaged anatomical part. The development of more refined estimations of the velocity field are actually an active research topic, especially for flows in the heart cavities. The only fast technique that we are aware of is based solely on the divergence equation ($\nabla\cdot u = 0$) on the imaging plane \cite{ohtsuki2006flow,uejima2010new}. Despite its rapidity and simplicity, the contribution of the out-of-plane velocity and the momentum conservation are neglected, leading to very non-physical approximations. Another existing alternative are classical inverse problems in which a joint state-parameter estimation is performed. Their main drawback is the large computational times, which is incompatible to the clinical practice.

\section{Reconstruction methods}
\label{sec:RecMeth}

We start by introducing the state reconstruction methods that we use in the present work. We focus on explaining alternative ways of building reduced basis that are better tailored for the task of state estimation than the usual reduced bases. Our presentation is done for noiseless measurements and assumes that there is no model error. As already brought up, this is due to the fact that we lack from real measurements for our targeted application. 

\subsection{State estimation and recovery algorithms}
\label{sec:SE}
Let $\Omega$ be a domain of $\bR^d$ for a given dimension $d\geq 1$ and let $V$ be a Hilbert space defined over $\Omega$, with inner product $\<\cdot, \cdot\>$ and norm $\Vert \cdot \Vert$. Our goal is to recover an unknown function $u\in V$ from $m$ measurement observations
\begin{equation}
\label{eq:meas}
\ell_i(u),\quad i=1,\dots,m,
\end{equation}
where the $\ell_i$ are linearly independent linear forms over $V$. In many applications, each $\ell_i$ models a sensor device which is used to collect the measurement data $\ell_i(u)$. If the observations come in the form of an image as in the application of this paper, each $\ell_i$ may represent the response of the system in a given pixel. The Riesz representers of the $\ell_i$ are denoted by $\omega_i$ and 
span an $m$-dimensional space
$$
\Wm={\rm span}\{\omega_1,\dots,\omega_m\} \subset V.
$$
The observations $\ell_1(u),\dots, \ell_m(u)$ are thus equivalent to knowing the orthogonal projection
\begin{equation}
\omega =P_\Wm u.
\end{equation}
In this setting, the task of recovering $u$ from the measurement observation $\omega$ can be viewed as building a recovery algorithm
$$
A:\Wm\mapsto V
$$
such that $A(P_{\Wm}u)$ is a good approximation of $u$ in the sense that $\Vert u - A(P_{\Wm}u) \Vert$ is small.

Recovering $u$ from the measurements $P_\Wm u$ is a very ill-posed problem since there are infinitely many $v\in V$ such that $P_\Wm v = \omega$. It is thus necessary to add some a priori information on $u$ in order to recover the state up to a guaranteed accuracy.
We are motivated by the setting where $u$ is a solution to some parameter-dependent PDE of the general form
\begin{equation*}
\cP(u, y) = 0,
\end{equation*}
where $\cP$ is a differential operator and $y$ is a vector of parameters that describes some physical property and lives in a given set $\rY\subset \bR^p$. Therefore, our prior on $u$ is that it belongs to the set
\begin{equation}
\label{eq:manifold}
\cM \coloneqq \{ u(y) \in V \; :\; y\in\rY \},
\end{equation}
which is sometimes referred to as the {\em solution manifold}. The performance of a recovery mapping $A$ is usually quantified in two ways:
\begin{itemize}
\item If the sole prior information is that $u$ belongs to the manifold $\cM$, the performance is usually measured by the worst case reconstruction error
$$
E_{\wc}(A,\cM) = \sup_{u\in\cM} \Vert u - A(P_\Wm u) \Vert \, .
$$
\item In some cases $u$ is described by a probability distribution $p$ on $V$ supported on $\cM$. This distribution is itself induced by a probability distribution on $\rY$ that is assumed to be known. In this Bayesian-type setting, the performance is usually measured in an average sense through the mean-square error
$$
E^2_{\ms}(A,\cM) = \bE\left( \Vert u - A(P_\Wm u) \Vert^2\right) = \int_V \Vert u - A(P_\Wm u)\Vert^2 dp(u) \, ,
$$
and it naturally follows that $E_{\ms}(A,\cM)\leq E_{\wc}(A,\cM) $.
\end{itemize}

\subsection{Linear and nonlinear algorithms using reduced modeling}
\label{sec:algos}
Reduced models are a family of methods that produce each a hierarchy of spaces $(V_n)_{n\geq 1}$  that approximate the solution manifold well in the sense that
\begin{equation*}
\eps_n \coloneqq \sup_{u\in\cM} \dist(u, V_n)  \,
,\quad \text{or} \quad
\delta^2_n \coloneqq \bE\left(  \dist(u, V_n)^2\right) \, 
\end{equation*}
decays rapidly as $n$ grows for certain classes of PDEs. Several methods exist to build these spaces among which stand the reduced basis method (see \cite{RHP2007}), the (Generalized) Empirical Interpolation Method (see \cite{BMNP2004, MM2013, MMT2016}), Proper Orthogonal Decomposition (POD, \cite{sirovich1987,berkooz1993}) and low-rank methods (see \cite{CDS2011, CD2015acta}).

Linear reconstruction algorithms that make use of reduced spaces $V_n$ are the Generalized Empirical Interpolation Method (GEIM) introduced in \cite{MM2013} and further analyzed in \cite{MMPY2015, MMT2016} and the Parametrized Background Data-Weak Approach (PBDW) introduced in \cite{MPPY2015} and further analyzed in \cite{BCDDPW2017}. Note that some modified versions have been proposed to address measurement noise (see, e.g., \cite{ABGMM2017, Taddei2017}) and  other recovery algorithms involving reduced modelling have also been recently proposed (see \cite{KGV2018}). 

\subsubsection{PBDW, a linear recovery algorithm}
\label{sec:linearPBDW}
Given a measurement space $\Wm$ and a reduced model $V_n$ with $1\leq n\leq m$, the PBDW algorithm $$A^\pbdw_{m,n}:\Wm \to V$$ gives for any $\omega \in \Wm$ a solution of
\begin{equation*}
\min_{u \in \omega + W^\perp } \dist(u, V_n).
\end{equation*}
Denoting
\begin{equation}
\label{eq:infsup}
\beta(X,Y):=\inf_{x\in X}\sup_{y\in Y}\frac {\<x,y\>}{\|x\|\, \|y\|}=\inf_{x\in X}\frac {\|P_{Y} x\|}{\|x\|} \in [0,1]
\end{equation}
for any pair of closed subspaces $(X,Y)$ of $V$, the above optimization problem has a unique minimizer 
\begin{equation}
\label{eq:uStar}
A^\pbdw_{m,n}(\omega) = u_{m,n}^*(\omega) \coloneqq \argmin_{u = \omega + W^\perp } \dist(u, V_n).
\end{equation}
as soon as $n\leq m$ and $\beta(\Vn,\Wm)>0$. We adhere to these two assumptions in the following.

As proven in section \ref{appendix:linear-pbdw}, an explicit expression of $u_{m,n}^*(\omega)$ is
\begin{equation}
\label{eq:uStarExplicit}
u_{m,n}^*(\omega) = v^*_{m,n}(\omega) + \omega - P_\Wm v^*_{m,n}(\omega)
\end{equation}
with
\begin{equation}
\label{eq:vStarExplicit}
v^*_{m,n}(\omega) = \( P_{\Vn | \Wm} P_{\Wm | \Vn} \)^{-1} P_{\Vn | \Wm}(\omega),
\end{equation}
where, for any pair of closed subspaces $(X,Y)$ of $V$, $P_{X | Y}: Y \to X$ is the orthogonal projection into $X$ restricted to $Y$. The invertibility of the operator $P_{\Vn | \Wm} P_{\Wm | \Vn}$ is guaranteed under the above conditions.

Formula \eqref{eq:uStarExplicit} shows that $A^\pbdw_n$ is a bounded linear map from $\Wm$ to $\Vn\oplus(W_m\cap V_n^\perp)$. Depending on whether $\Vn$ is built to address the worst case or mean square error, the reconstruction performance is bounded by
\begin{equation}
\label{eq:err-wc-pbdw}
e_{m,n}^\wcpbdw = E_{\wc}(A^\pbdw_{m,n}, \cM) \leq \beta^{-1}(V_n, \Wm) \max_{u\in \cM}\dist(u, V_n\oplus(V_n^\perp\cap\Wm)) \leq \beta^{-1}(V_n, \Wm) \, \eps_n,
\end{equation}
or
\begin{align}
e_{m,n}^\mspbdw = E_{\ms}(A^\pbdw_{m,n}, \cM)
&\coloneqq \bE\left( \Vert u - A^{(\text{pbdw})}_n(P_\Wm u)\Vert^2\right)^{1/2} \nonumber\\
&\leq \beta^{-1}(V_n, \Wm) \bE\left(\dist(u, V_n\oplus(V_n^\perp\cap\Wm)^2\right)^{1/2}  \nonumber\\
&\leq \beta^{-1}(V_n, \Wm) \, \delta_n, \label{eq:err-ms-pbdw}
\end{align}
Note that $\beta(\Vn, \Wm)$ can be understood as a stability constant. It can also be interpreted as the cosine of the angle between $\Vn$ and $\Wm$. The error bounds involve the distance of $u$ to the space $V_n\oplus(V_n^\perp\cap\Wm)$ which provides slightly more accuracy than the reduced model $\Vn$ alone. In the following, to ease the reading we will write errors only with the second type of bounds that do not involve this correction part on $V_n^\perp\cap\Wm$.

An important observation is that for a fixed measurement space $W_m$ (which is the setting in our numerical tests), the errors $e_{m,n}^\wcpbdw$ and $e_{m,n}^\mspbdw$ reach a minimal value $e_{m,n^*_\wc}^\wcpbdw$ and $e_{m,n^*_\ms}^\mspbdw$ as the dimension $n$ varies from 1 to $m$. This behavior is due to the trade-off between the increase of the approximation properties of $V_n$ as $n$ grows and the degradation of the stability of the algorithm, given here by the decrease of $\beta(V_n, \Wm)$ to 0 as $n\to m$. As a result, the best reconstruction performance with  PBDW is given by
$$
e_{m,n^*_\wc}^\wcpbdw = \min_{1\leq n \leq m} e_{m,n}^\wcpbdw,
\quad\text{or}\quad
e_{m,n^*_\ms}^\mspbdw = \min_{1\leq n \leq m} e_{m,n}^\mspbdw.
$$

We finish this section with several remarks:
\begin{enumerate}
\item We do not consider measurement noise and model error. However, the PBDW algorithm can correct to some extent the model misfit (through the term $\eta^*$, see \ref{appendix:linear-pbdw}). Also, some extensions of the current setting have been proposed to address measurement noise (see \cite{Taddei2017, ABGMM2017, GMMT2019}).
\item \om{In our application, the manifold $\cM$ will be a family of incompressible fluid flow solutions of a parametric incompressible Navier-Stokes equation. In this case, $V_n$ is usually built such that all functions $v\in V_n$ satisfy the divergence-free condition $\nabla\cdot v=0$. As a result, a reconstruction with only $v^*_n$ (see \eqref{eq:vStarExplicit}) will yield a divergence-free approximation of the flux. Note however that  the full PBDW reconstruction $u^*_n$ from equation \eqref{eq:uStarExplicit} does not guarantee this property since the component $ \omega - P_\Wm v^*_{m,n}(\omega)$ may not be divergence free. The mass conservation of the reconstruction could be enforced, for instance, by considering its projection on divergence-free fields similar to the one used in fractional methods (like Chorin-Temam) for fluid flows.}
\item Note that in the present setting the measurement space $\Wm$ is fixed and we will adhere to this assumption in the rest of the paper. This is reasonable since usually the nature and location of the sensors is fixed in our application following the experience of the medical doctors. A different, yet related problem, would be to optimize the choice of the measurement space $\Wm$ \om{and we refer to \cite{MMT2016, BCMN2018, Ben, Ai,CK,YS} for works on this topic.}
\end{enumerate}

\subsubsection{Alternative reconstruction algorithms}
\label{sec:nonlinear}
Taking PBDW as a starting point, we next describe two different strategies to build recovery algorithms. They all incorporate an affine extension of PBDW that we explain next. One motivation to introduce it is because it can give more accuracy when the solution manifold $\cM$ is not localized near the origin. This typically happens when the state $u$ is a perturbation of a nominal state $\bar u \in V$.

\textbf{Affine PBDW:} Given an average or nominal state $\bar u \in V$, we can formulate the equivalent affine version of PBDW, which reads

\begin{equation}
\label{eq:aff-pbdw}
A_{m,n}^\affpbdw(\omega) \coloneqq \argmin_{u \,\in\, \omega + W^\perp } \dist(u, \bar u+ V_n),
\end{equation}
As shown in \ref{appendix:linear-pbdw}, this algorithm can also be written as
\begin{equation*}
A_{m,n}^\affpbdw(\omega) = \bar u + A_{m,n}^\pbdw(\omega-\bar \omega),
\end{equation*}
where $\bar \omega = P_{W_m} \bar u$ and $A_{m,n}^\pbdw$ is the linear PBDW algorithm of the previous section. We can easily see through the last expression that $A_{m,n}^\affpbdw$ is an affine mapping from $\Wm$ to $\bar u + V_n \oplus (\Wm\cap V_n^\perp)$. An interesting feature of this affine extension is that if $\omega=0$, the algorithm yields a nonzero reconstruction in $\Wm^\perp$ based on the nominal state $\bar u$, $A_{m,n}^\affpbdw(0)= \bar u - A_{m,n}^\pbdw(\bar \omega)$. This is in contrast to the linear version where $A_{m,n}^\pbdw(0)=0$.

Proceeding similarly as before, for a given $u\in\cM$, the error is bounded by (see \cite{CDDFMN2019})
\begin{align}
\Vert u - A_{m,n}^\affpbdw(\omega) \Vert
&\leq \beta^{-1}(\Vn, W) \dist(u, \bar u + \Vn\oplus(\Wm\cap\Vn^\perp))  \nonumber \\
&\leq \beta^{-1}(\Vn, W) \dist(u, \bar u + \Vn) \label{eq:uStarAffineBound}
\end{align}
and
\begin{equation}
\label{eq:err-wc-aff-pbdw}
e_{m,n}^\wcaff \coloneqq E_{\wc}(A^\affpbdw_{m,n}, \cM) \leq \beta^{-1}(V_n, \Wm) \, \eps^\affpbdw_n,
\end{equation}
or
\begin{equation}
e_{m,n}^\msaff \coloneqq E_{\ms}(A^\affpbdw_{m,n}, \cM) \coloneqq \bE\left( \Vert u - A^\affpbdw_{m,n}(P_\Wm u)\Vert^2\right)^{1/2}\leq \beta^{-1}(V_n, \Wm) \, \delta_n^\affpbdw,
\end{equation}
where
\begin{equation*}
\eps^\affpbdw_n \coloneqq \sup_{u\in\cM} \dist(u, \bar u + V_n)  \,
,\quad \text{or} \quad
(\delta^\affpbdw_n)^2 \coloneqq \bE\left(  \dist(u, \bar u + V_n)^2\right) \, 
\end{equation*}
In the following, to simplify notation, we will use $\eps_n$ and $\delta_n$ to denote either the error in the linear PBDW or its affine version since the reasoning and the estimates that will be derived next have the same form.

\textbf{Partition of $\cM$:} Our first strategy stems from the fact that we may have very few observations ($m$ small) and since we must have $n\leq m$ for reconstruction, we may not have enough approximation power in the interval where $n$ can range. Also, the approximation errors $(\eps_n)_n$ or $(\delta_n)_n$ provided by reduced basis may not always decrease rapidly to zero (the Kolmogorov $n$-width of $\cM$ may decrease slowly). However, the physical structure of the problem could give a natural decomposition of the manifold $\cM$
into different subdomains $\cM^{(k)}$ that are better adapted for model
reduction in the sense that the errors $(\eps^{(k)}_n)_n$ or
$(\delta^{(k)}_n)_n$ may decrease faster. Works based on that partition principle have been developed in the context of model reduction of forward problems \cite{carlberg2015adaptive,amsallem2016pebl,amsallem2012nonlinear,
peherstorfer2014localized}) and we propose to adapt it for data assimilation. A simple setting where we can decompose the manifold in our context is when we work in an application for which it is possible to know exactly a subset of entries in y for any target $u(y)$ during the online phase, 
say $\overline{y} \in \mathbb{R}^{\overline{p}}$ (with
$\overline{p}<p$). Given $\overline{y}$ and a data-base generated from
the governing PDE we may produce a disjoint union of $K$
subsets $\rY^{(k)}$ which yields a decomposition of $\cM$ into subsets
$\cM^{(k)}=u(\rY^{(k)})$. We can thus build reduced models
$(V_n^{(k)})_{n=1}$ for each subset $\cM^{(k)}$ and then reconstruct
with the linear or affine PBDW. 

Proceeding similarly as in the previous
section, the reconstruction performance on subset $\cM^{(k)}$ is

\begin{equation*}
e_{m,n}^\wcaffk = E_{\wc}(A^\aff_{m,n}, \cM^{(k)}) \leq \beta^{-1}(V_n^{(k)}, \Wm) \, \eps^{(k)}_n,
\end{equation*}
or
\begin{equation*}
e_{m,n}^\msaffk = E_{\ms}(A^\aff_{m,n}, \cM^{(k)}) \coloneqq \bE\left( \Vert u - A^\aff_{m,n}(P_\Wm u)\Vert^2\right)^{1/2}\leq \beta^{-1}(V_n^{(k)}, \Wm) \, \delta^{(k)}_n.
\end{equation*}
The best reconstruction performance for $\cM^{(k)}$ is thus 
$$
e_{m,n^*_\wc(k)}^\wcaffk = \min_{1\leq n \leq m} e_{m,n}^\wcaffk,
\quad\text{or}\quad
e_{m,n^*_\ms(k)}^\msaffk = \min_{1\leq n \leq m} e_{m,n}^\msaffk.
$$
It follows that the performance in $\cM=\cup_{k=1}^K \cM^{(k)}$ is
$$
e_{m}^\wcaff = \max_{1\leq n \leq m} e_{m,n^*_\ms(k)}^\wcaffk,
\quad\text{or}\quad
e_{m,n^*_\ms(k)}^\msaffk = \sum_{k=1}^K \omega_k e_{m,n^*_\ms(k)}^\msaffk,
$$
where $\omega_k = p( u\in \cM^{(k)} )$. 

\textbf{Data-based reduced models:} The second strategy is motivated by the observation that the reduced models $V_n$ or $V_n^{(k)}$ of the previous approaches are built independently of the given measurement space $W_m$ and also of the given measurement observation $\omega$. For a given $\omega\in \Wm$, we can build a data-driven $V_n(\omega)$ with an orthogonal matching pursuit greedy algorithm (OMP) that we explain next. Once this reduced model has been computed, we reconstruct with the data-driven affine version of PBDW,
\begin{equation}
\label{eq:uStarAffine}
A_{m,n}^\dd(\omega) \coloneqq \argmin_{u \,\in\, \omega + W^\perp } \dist(u, \bar u+ V_n(\omega)),
\end{equation}
where the difference with respect to \eqref{eq:aff-pbdw} is that now $\Vn$ depends on $\omega$. The reconstruction performance of this algorithm is bounded by
$$
e_{m,n}^\wcdd = 
E_\wc(A_{m,n}^\dd,\cM) \leq \sup_{u\in\cM}  \beta^{-1}(\Vn(P_\Wm u), W) \dist(u, \bar u + \Vn(P_\Wm u)),
$$
in the worst case setting. Similarly as before, $e_{m,n}^\wcdd$ reaches a minimum $e_{m,n^*_\dd}^\wcdd$ when $n$ varies from 1 to $m$. Since $V_n$ is now adapted to the measurement observations, we expect that the current algorithm performs better than its classical linear counterpart.

We now present the OMP greedy algorithm that we propose. Let 
$$
\cD \coloneqq \{ v = u / \Vert u \Vert \ :\ u \in \cM \}
$$
be the set of normalized functions from $\cM$.
If $\bar u =0$, the first element $\varphi_1$ is chosen as 

\begin{equation}
\label{eq:omp1}
\varphi_1= \frac{1}{\# \cD} \sum_{u \in \cD} u.
\end{equation}

For $n>1$, given $V_{n}=\mathrm{span}\{\varphi_1, \ldots, \varphi_n\}$, we select
\begin{equation}
\label{eq:omp2}
\varphi_{n+1}\in \argmax_{v\in \cD} \left | \left\< w - P_{P_{\Wm} V_n} w, \frac{P_{\Wm} v}{\|P_{\Wm} v\|} \right\> \right |
\end{equation}
where $P_{\Wm} V_n = \mathrm{span}\{P_{\Wm} \varphi_1,\ldots,P_{\Wm} \varphi_n\}$. We set $V_{n+1}=\mathrm{span}\{V_n, \varphi_{n+1}\}$.

Note that all operations in this algorithm are done in the space $\Wm$. Hence we can do all calculations in $\mathbb{R}^m$, which makes this algorithm be very fast since it does not involve  computations with functions from the whole space $V$ (see \ref{appendix:OMP} for further details on its implementation).

In the case $\bar u \neq 0$, we introduce $\bar \omega = P_\Wm \bar u$ and the shifted set
$$
\delta_{\bar u} \cD = \left\lbrace v = \frac{u - \bar u}{\Vert u - \bar u \Vert} \; : \; u \in \cM,\, u\neq \bar u \right\rbrace .
$$
Now it suffices to apply the previous greedy algorithm to the target function $\omega -\bar \omega$ instead of $w$ and do the search over $\delta_{\bar u} \cD$ instead of $\cD$.

\begin{remark}
Note that both nonlinear approaches are very different in nature: the first is based on localization and the second is purely data-driven. Their performance will thus be very dependent on the physical problem under consideration and the measurement setting: the nature of $\Wm$, the nature of the actual measurements $\omega=P_\Wm u$, the nature of the manifold and its alignment with $\Wm$ locally at $\omega$. Note also that both approaches can be combined and we can consider a piece-wise data-driven algorithm (its performance has actually been tested in the numerical tests below). 
\end{remark}

\section{\om{Computational details of the methods}}
\label{sec:impl}
\subsection{Discretizations}
In the previous section and for the sake of clarity, we have given an idealized description of PBDW and the other reconstruction methods where we do not mention certain discretization aspects that inevitably come into play. To start with, the snapshot functions $u\in \cM$ cannot be computed exactly but only up to some tolerance. One typical instance (which is used in our work) is the finite element method which gives an approximation $u_h\in V_h$ of $u\in V$ where $V_h$ is a finite element space with $\cN$ degrees of freedom. Therefore the computable states are  elements of the perturbed manifold
$$
\cM_h \coloneqq \{ u_h(y) \in V_h \, :\, y \in \rY \}.
$$
The computation of reduced models $V_n$ involves a finite training subset $\widetilde\cM_{\train} \subset \cM_h$ of snapshots, and we denote by $\#\widetilde\cM_{\train} $ its cardinality. As a consequence, all the reduced models are low dimensional subspaces of $V_h$. Note that the fact that the true states do not belong to $\cM_h$ can be interpreted as a model bias.

The methodology also requires the computation of the Riesz lifts $\omega_i$ in order to define the measurement space $W_m$. Since we work in the space $V_h$, these can be defined
as elements of this space satisfying
$$
\<\omega_i,v\>=\ell_i(v), \quad v\in V_h,
$$
thus resulting in a measurement space $W\subset V_h$.

In what follows, we keep our idealized discussion and do not systematically recall  that all computations take place in a background discretization space $V_h$.

\subsection{Algebraic formulation of PBDW}
\label{appendix:linear-pbdw}

Here we explicitly derive the algebraic formulation of $u_{m,n}^*(\omega)$, the function given by the linear PBDW algorithm. Let $X$ and $Y$ be two finite dimensional subspaces of $V$ and let
\begin{align*}
P_{X | Y} : Y &\to X \\
y &\mapsto P_{X | Y} (y)
\end{align*}
be the orthogonal projection into $X$ restricted to $Y$. That is, for any $y\in Y$, $P_{X | Y}(y)$ is the unique element $x\in X$ such that 
$$
\left< y - x, \tilde x\right>= 0,\quad  \forall \tilde x \in X.
$$

\begin{lemma}
Let $\Wm$ and $\Vn$ be an observation space and a reduced basis of dimension $n\leq m$ such that $\beta(\Vn, \Wm)>0$. Then the linear PBDW algorithm is given by
\begin{equation}
\label{eq:explicit-u}
u_n^*(\omega) = \omega + v^*_n - P_W v^*_n,
\end{equation}
with
\begin{equation}
\label{eq:explicit-v}
v^*_n = \( P_{V_n | \Wm} P_{\Wm | V_n} \)^{-1} P_{V_n | \Wm} (\omega).
\end{equation}
\end{lemma}
\begin{proof}
By formula \eqref{eq:uStar}, $u^*_n(\omega)$ is a minimizer of
\begin{align*}
\min_{u \in \omega + \Wm^\perp } \dist(u, V_n)^2
&= \min_{u \in \omega + \Wm^\perp } \min_{v\in V_n} \Vert u - v\Vert^2 \\
&= \min_{v\in V_n}  \min_{\eta \in \Wm^\perp} \Vert \omega + \eta - v \Vert^2 \\
&= \min_{v\in V_n} \Vert \omega -v - P_{\Wm^\perp}(w -v) \Vert^2 \\
&= \min_{v\in V_n} \Vert \omega -v + P_{\Wm^\perp}(v) \Vert^2 \\
&= \min_{v\in V_n} \Vert \omega - P_{\Wm}(v) \Vert^2 .
\end{align*}
The last minimization problem is a classical least squares optimization. Any minimizer $v^*_n\in V_n$ satisfies the normal equations
$$
P^*_{\Wm|V_n}  P_{\Wm|V_n} v^*_n = P^*_{\Wm|V_n} w,
$$
where $P^*_{\Wm|V_n} : V_n \to \Wm$ is the adjoint operator of $P_{\Wm|V_n}$. Note that $P^*_{\Wm|V_n}$ is well defined since $\beta(V_n, \Wm)= \min_{v\in V_n} \Vert P_{\Wm|V_n} v \Vert / \Vert v \Vert >0 $, which implies that $P_{\Wm|V_n}$ is injective and thus admits an adjoint. Furthermore, since for any $w \in \Wm$ and $v\in V_n$, $\< v, w \>=\< P_{\Wm|V_n} v , w \> = \< v , P_{V_n|\Wm}w \>$, it follows that $P^*_{\Wm|V_n} = P_{V_n|\Wm}$, which finally yields that the unique solution of the least squares problem is
$$
v^*_n = \( P_{V_n | \Wm} P_{\Wm | V_n} \)^{-1} P_{V_n | \Wm} w .
$$
Therefore $u^*_n = w + \eta^*_n = w + v^*_n - P_\Wm v^*_n$.
\end{proof}

As a direct consequence of this Lemma, we derive the following formulation for the affine PBDW algorithm.
\begin{corollary}
Let $\Wm$ and $\Vn$ be an observation space and a reduced basis of dimension $n\leq m$ such that $\beta(\Vn, \Wm)>0$. Then the affine PBDW algorithm with respect to a nominal state $\bar u$ is given by
\begin{equation}
A_{m,n}^\affpbdw(\omega) = \bar u + u^*_{m,n}(w-\bar w)
\end{equation}
where $\bar w = P_\Wm \bar u$ and $u^*_{m,n}$ is the reconstruction with the linear PBDW method.
\end{corollary}
\begin{proof}
The result follows by following the same lines as in the linear case for the shifted minimization problem
\begin{align*}
\min_{u \in \omega + \Wm^\perp } \dist(u, \bar u + V_n)^2
&= \min_{u \in \omega + \Wm^\perp } \min_{v\in V_n} \Vert u - \bar u - v\Vert^2 \\
\end{align*}
\end{proof}

We next derive the symmetric linear system of equations to be solved in order to compute $v^*_n$ in expression \eqref{eq:explicit-v}. Let $F$ and $H$ be two finite-dimensional spaces of $V$ of dimensions $n$ and $m$ respectively in the Hilbert space $V$ and let $\cF=\{f_i\}_{i=1}^n$ and $\cH=\{h_i\}_{i=1}^m$ be a basis for each subspace respectively. The Gram matrix associated to $\cF$ and $\cH$ is
$$
\bG(\cF, \cH) = \left(  \left< f_i, h_j\right> \right)_{\substack{1\leq i \leq n \\ 1\leq j \leq m}}.
$$
These matrices are useful to express the orthogonal projection
$P_{ F | H}: H\mapsto F$ in the bases $\cF$ and $\cH$ in terms of the matrix
$$
\bP_{F | H} = \bG(\cF, \cF)^{-1} \bG(\cF, \cH).
$$
As a consequence, if $\cV_n = \{ v_i \}_{i=1}^n$ is a basis of the space $V_n$ and $\cW_m = \{\omega_i\}_{i=1}^m$ is the basis of $W_m$ formed by the Riesz representers of the linear functionals $\{\ell_i\}_{i=1}^m$, the coefficients $\textbf{v}^*_n$ of the function $v^*_n$ in the basis $\cV_n$ are the solution to the normal equations
\begin{equation}
\label{eq:normaleq}
\bP_{V_n | W_m} \bP_{W_m | V_n} 
\textbf{v}^*_n =  \bP_{V_n | W_m} \bG(\cW_m, \cW_m)^{-1} \textbf{w},
\end{equation}
where
$$
\bP_{V_n | W_m} = \bP_{V_n | W_m}^T
$$
since $P^*_{\Wm|V_n} = P_{\Wm|V_n}$ and $\textbf{w}$ is the vector of measurement observations
$$
\textbf{w} = (\left< u, \omega_i\right>)_{i=1}^m.
$$
Usually $\textbf{v}^*_n$ is computed with a QR decomposition or any other suitable method. Once $\textbf{v}^*_n$ is found, the vector of coefficients $\textbf{u}_n^*$ of $u^*_n$ easily follows.

\subsection{Practical implementation of OMP}
\label{appendix:OMP}

Introducing the set $\cW = \{ P_{\Wm} v \, : \, v \in \cD \}$, the recursive step of the OMP algorithm \eqref{eq:omp2} can equivalently written in terms of functions of $\Wm$ as follows. For $n>1$, given $V_{n}=\mathrm{span}\{\varphi_1, \ldots, \varphi_n\}$, we select (see equation \eqref{eq:omp2})
\begin{equation}
z_{n+1}\in \argmax_{z\in \cW} \left | \left\< \omega - P_{P_{\Wm}(V_n)} \omega, \frac{z}{\| z \|} \right\> \right |.
\end{equation}
For the chosen $z_{n+1}$, we take one of the corresponding functions $\varphi_{n+1}$ from $\cD$ that satisfy $P_{\Wm} \varphi_{n+1}=z_{n+1}$. We then set $V_{n+1}=\mathrm{span}\{V_n, \varphi_{n+1}\}$.
The computation of $P_{P_{\Wm} V_n} \omega$ is a straightforward solving of a linear problem. We look for the coefficients $c = (c_i)_{i=1}^n$ such that $P_{P_{\Wm} (V_n)} \omega = \sum_{i=1}^n c_i P_{\Wm} \varphi_i$. Since
\begin{equation}
  \innerp{P_{P_{W_m}(V_n)} \omega}{P_{W_m} \varphi_i} = \innerp{\omega}{P_{W_m} \varphi_i } ,\quad \forall i=1,\dots, n
  \label{eq:calculating_P_omega}
\end{equation}
it follows that
$$
\sum_{j=1}^n c_j \left< P_\Wm \varphi_i , P_\Wm \varphi_j \right> = \left< \omega, P_\Wm \varphi_i \right>,\quad 1\leq i \leq n,
$$
which is a linear system of the form
\begin{equation}
A^{\text{OMP}} c = g^{\text{OMP}}.
\label{eq:omp_offline}
\end{equation}
Denoting $\{ \rm w_1,\dots,w_m \}$ an orthonormal basis of $\Wm$, we have
$$
A^{\text{OMP}} = ( a_{i,j}  )_{1\leq i, j \leq n},\quad a_{i,j} = \left< P_\Wm \varphi_i , P_\Wm \varphi_j \right> = \sum_{k=1}^m \left< \rm{w}_k, \varphi_j\right> \left< \rm{w}_k, \varphi_i \right>
$$
and
$$
g^{\text{OMP}} = (g_i)_{i=1}^n,\quad g_i = \left < \omega, P_\Wm \varphi_i \right> = \sum_{k=1}^m \left< u, \rm{w}_k\right> \left< \varphi_i, \rm{w}_k \right>.
$$

\subsection{Summary of the methods investigated and their computational cost.}
\label{sec:summary-methods}
\paragraph{\textbf{Methods:}} The methods that we have implemented in our numerical tests and their computational cost are:
\begin{enumerate}
\item \emph{Linear PBDW (labelled \pl):} We build the linear spaces $V_n$ from a classical Proper Orthogonal Decomposition of the whole training set $\widetilde\cM_{\train}$. This is the classical PBDW approach explained in section \ref{sec:linearPBDW} and we consider it as the reference that our proposed methods should outperform. 

\item \emph{Nonlinear algorithm with manifold partitioning:} We partition the training set $\widetilde\cM_{\train}$ in a number of non-intersecting subsets (see section \ref{sec:nonlinear}). For each subset  $\widetilde \cM^{(k)}$ of the partition, we build in an offline phase a reduced model. Two constructions have been tested:
\begin{itemize}
\item \ppa: A Proper Orthogonal Decomposition on each partition $\widetilde \cM^{(k)}$. 
\item \pga: A greedy algorithm: for $n=1$, we set $V^{(k)}_1=\vspan\{u^{(k)}_1\}$ with
$$
u^{(k)}_1= \frac{1}{\# \widetilde\cM^{(k)}} \sum_{u \in \widetilde\cM^{(k)}} u.
$$

For $n>1$, we select 
\begin{equation}
\label{eq:pure-greedy}
u^{(k)}_n \in \argmax_{u\in \widetilde\cM^{(k)}} \Vert u - P_{V_{n-1}} u \Vert,
\end{equation}
and set $V^{(k)}_n = \vspan\{V^{(k)}_{n-1}, u^{(k)}_n\}$.

\end{itemize}

\item \pdba, \emph{Data-driven nonlinear algorithm with manifold partitioning:}  Since each ultrasound image can be seen as an observation $\omega\in\Wm$, we run the OMP algorithm of section \ref{sec:nonlinear} to build $V_n(\omega)$ and do the reconstruction. Note that the greedy search has to be done online since we need the knowledge of the measurement. To speed-up computations, instead of searching in the whole training set
$$
\delta_{\bar u} \widetilde \cD = \left\lbrace v = \frac{u - \bar u}{\Vert u - \bar u \Vert} \; : \; u \in \widetilde\cM_{\train} \right\rbrace ,
$$
we restrict the search to the partition
$$
\delta_{\bar u}\widetilde \cD^{(k)} = \left\lbrace v = \frac{u - \bar u}{\Vert u - \bar u \Vert} \; : \; u \in \widetilde\cM^{(k)} \right\rbrace .
$$
\end{enumerate}

\paragraph{\textbf{Computational costs:}} We outline the computational cost of the methods in Table \ref{tab:cost}. Note that all of them can be decomposed into an online and an offline phase. The cost of generating the snapshots database is of the same order in all methods, namely $\mathcal{O}(\mathcal{N}^\alpha \#\widetilde\cM_{\train})$ where $\alpha$ is the scaling behavior of the linear solver (usually $\alpha =2$ for iterative methods). The cost of building the reduced model in $\pl$ and $\ppa$ is the one to compute the SVD of the correlation matrix of the snapshots. For \pl, the matrix is of size $\#\widetilde\cM_{\train}^2$. For \ppa, we have to compute the SVD of $K$ matrices of size $(\#\widetilde\cM^{(k)})^2$ for $k=1\dots, K$. If these SVD computations are done in parallel, we can obtain important time reductions compared to \pl~since in general $\#\widetilde\cM^{(k)} \ll \#\widetilde\cM_{\train}$. Regarding the cost of the online phase, if we store the QR decomposition of normal equations \eqref{eq:normaleq} (at a cost of $\ord(n^3)$ operations, not reported on the table), the first three methods only need to solve the associated linear system, which costs $\ord(n^2)$ operations. In the case of \pdba, we additionally have to run a greedy algorithm in $\bR^m$ as explained in section \ref{appendix:OMP}, hence the additional cost reported in the table. Note however that the time required for this extra step is not significantly slowing down computations thanks to the fact that we work in $\bR^m$ and $m$ is usually moderate. This is in contrast to the greedy algorithm of the offline phase of \pga, which takes place in $\bR^\cN$.

\begin{table}
\centering
\begin{tabular}{|p{3cm}|c|c|c|c|}
\hline
\backslashbox{Cost}{Method}
&\pl
&\ppa
&\pga
&\pdba
\\\hline\hline

Offline: Database generation
& $\ord(\mathcal{N}^\alpha \#\widetilde\cM_{\train})$
&  $\ord(\mathcal{N}^\alpha \#\widetilde\cM_{\train})$
& $\ord(\mathcal{N}^\alpha \#\widetilde\cM_{\train})$
& $\ord(\mathcal{N}^\alpha \#\widetilde\cD_{\train})$ \\\hline

Offline: Construction reduced model
& $\mathcal{O}(\mathcal{N} \#\widetilde\cM_{\train}^2)$
& $\mathcal{O}(\mathcal{N} (\#\widetilde\cM^{(k)}_{\train})^2)$
& $\mathcal{O}(n\mathcal{N} \#\widetilde\cM^{(k)}_{\train})$
& $\cO(\cN \#\widetilde\cD_{\train}^{(k)} + m^3)$\\\hline

Online: Reconstruction
& $\ord(n^2 + n \cN)$
& $\ord(n^2 + n \cN)$
& $\ord(n^2 + n \cN)$
& $\ord(n^2 + n m \#\widetilde\cD^{(k)}_{\train} + n \cN)$\\\hline
\end{tabular}
\caption{\label{tab:cost} Computational cost of each method.}
\end{table}

\section{Application: Reconstruction of 3D blood velocity fields from Doppler ultrasound images}
\label{sec:numerical_example}

We apply the above described methodology to reconstruct a 3D blood velocity field on a human carotid artery from Doppler ultrasound images. The images are synthetically generated and the use of data from real patients is deferred to a future work. The main goal of the tests is twofold:
\begin{enumerate}
\item The first goal is to compare the different strategies to construct the space $V_n$ and the recovery algorithms. The method that can be consider a sort of a baseline to be compared to is the POD.
\item Second, we aim at assessing the ability of these recovery strategies to estimate the velocity in a semi-realistic idealised setting. Methods failing in achieving a $10\%$ accuracy on the peak velocity in this setting should not be used in more realistic scenarios. 
\end{enumerate}
The section is organized as follows. First, we present the parameter-dependent model that will define the manifold $\cM$ on which we will rely to compute different reduced models. Second, we explain how to define a measurement space $\Wm$ from a Doppler velocity image. Finally, we present results on the comparison of the different methods based on their reconstruction performance.

\subsection{The model: incompressible Navier-Stokes equations}
\label{subec:model}
Let $\Omega$ be a spatial bounded domain of $\mathbb{R}^3$ with the shape of a human carotid artery as given in Figure \ref{fig:geometry}. The boundary $\Gamma\coloneqq \partial \Omega$ is the union of the inlet part $\Gamma_i$ where the blood in entering the domain, the outlets $\Gamma_{o,1}$ and $\Gamma_{o,2}$ where the blood is exiting the domain after a bifurcation, and the walls $\Gamma_w$.


\begin{figure}[!htbp]
\centering
\subfigure[Domain $\Omega$ used in the simulations. Note the small stenosis in the upper part of the bifurcation. \label{fig:geometry}]{ 
\includegraphics[width=0.55\textwidth]{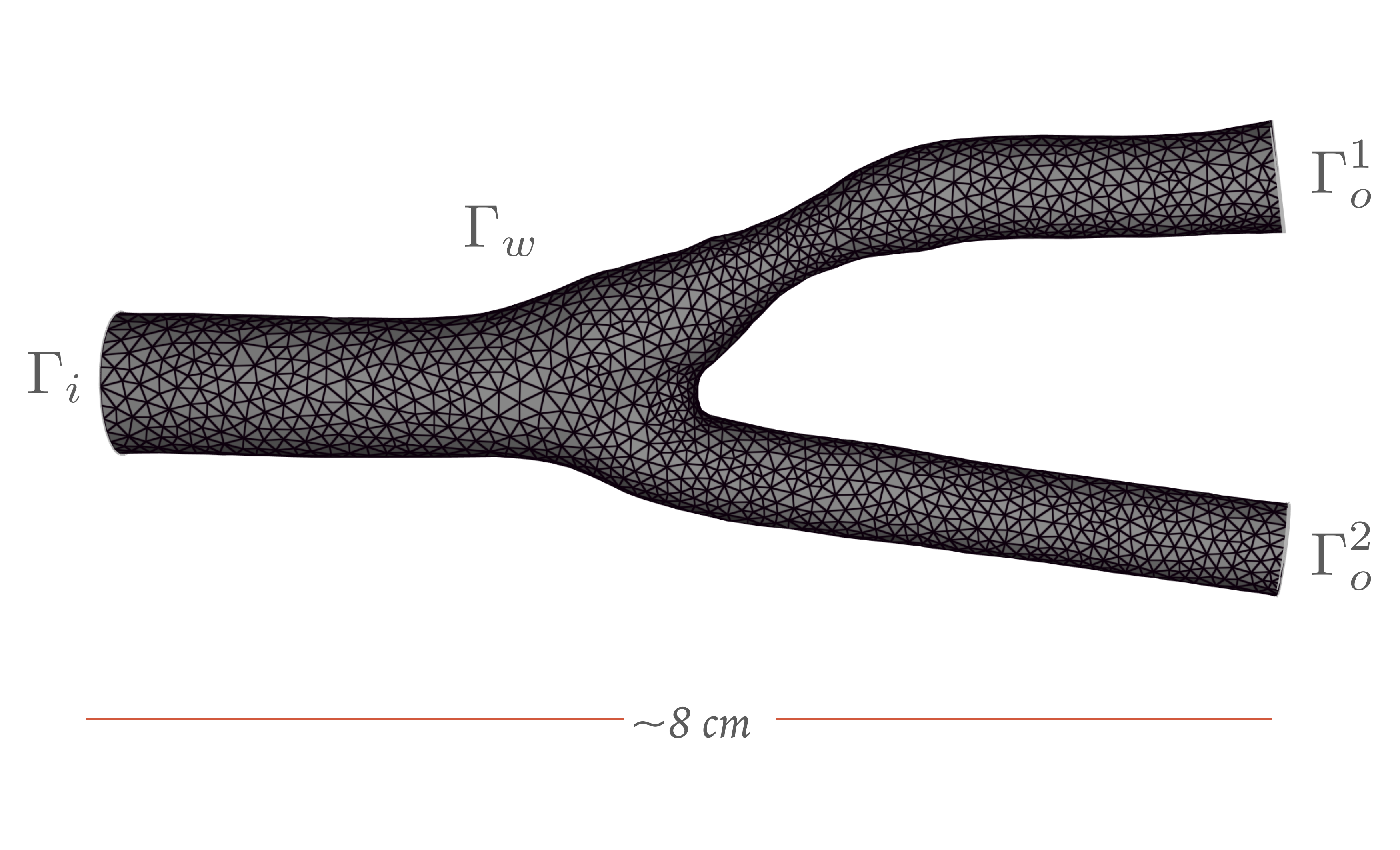}}
\subfigure[The function $g(t)$ for the inlet boundary condition. \label{fig:q_inflow_common_carotid}]{ 
\includegraphics[width=0.43\textwidth]{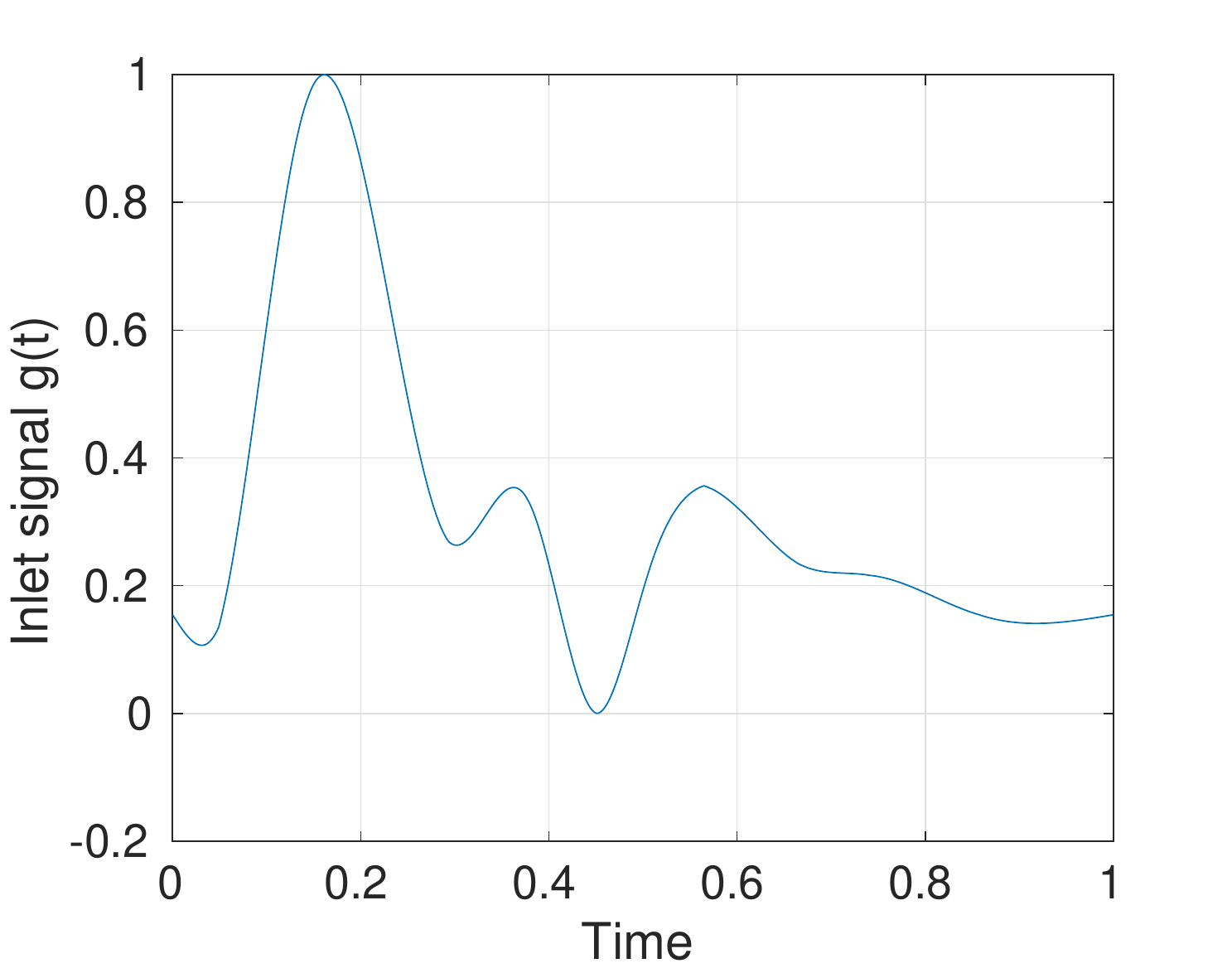}}
\caption{Geometry and inlet boundary function of the test case.}
\end{figure}

We consider the following incompressible Navier-Stokes equations on $\Omega$ and over the time interval $[0,T]$ for $T>0$. For a fluid with density $\rho\in\mathbb{R}^+$ and dynamic viscosity $\mu\in\mathbb{R}^+$, we search for all $t\in [0,T]$ the couple $(u(t),p(t)) \in [H^1(\Omega)]^3 \times L^2(\Omega)$ of velocity and pressure such that

\begin{equation}
\left\lbrace
\begin{aligned}
\rho \frac{\partial u}{\partial t}(t)  + \rho u(t) \nabla u(t) - \mu \Delta u(t) + \nabla p(t) = 0,\quad\text{ in } \Omega  \\ 
\nabla \cdot u (t) = 0,\quad \text{ in } \Omega.
\end{aligned}
\right.
\label{eq:Navier_Stokes}
\end{equation}
These equations are closed by adding a zero initial condition and the following boundary conditions:
\begin{itemize}
\item Boundary conditions:
\begin{itemize}
\item No-slip condition for the vessel wall, that is, $u = (0,0,0)^T$ on $\Gamma_w$.
\item The inlet boundary $\Gamma_i$ lies in the $xz$ plane and we apply a Dirichlet condition $u = [0,u_{\text{in}},0]^T$. The component b is a function $u_{\text{in}}(t, x, z) = u_0~ g(t) f(x, z)$, where:
\begin{itemize}
\item $u_0\in \mathbb{R}^+$ is an scaling factor. The function $g(t)$ is built by interpolating experimental flow data in the common carotid area taken in \cite{boundaryConditionsFromData}. Its behavior is given in Figure \ref{fig:q_inflow_common_carotid}.
\item The function $f$ is a 2D logit-normal distribution
\begin{equation}
  f(x) = \frac{1}{x(1-x)z(1-z)} \exp\left\lbrace-0.5 \left(\log \left(\frac{x}{1-x}\right)  - s  \right)^2 - 0.5
    \left(\log \left(\frac{z}{1-z}\right) \right)^2\right\rbrace,
  \label{eq:logit_normal_inlet}
\end{equation}
where the parameter, $s\in \mathbb{R}^+$, controls the axial symmetry of the inlet flow. \\
\end{itemize}


\item For the outlet boundaries $\Gamma_{o,1}$ and $\Gamma_{o,2}$, we use the so-called Windkessel model (see  \cite{Formagia_Cardiovascular_Mathematics}), which gives the average pressure over each $\Gamma_{o,k}$,
$$
\bar p_{o,k} = p_d^k + R_p^k \int_{\Gamma_{o,k}} u \cdot n ,\quad k=1,\,2
$$

where $p_d^k \in \mathbb{R}$ is called \emph{distal pressure} and is the solution to the ordinary differential equation:
\begin{equation}
\label{eq:pd}
\begin{cases}
C_d^k \dfrac{d p_d^k}{d t} + \dfrac{p_d^k}{R_d} &= \int_{\Gamma_{o,k}} u \cdot n  \\
p_d^k(t=0) &= p_{d, k} \text{ given.}
\end{cases}
\end{equation}

This model aims to represent the cardiovascular system behavior beyond the boundaries of
the working domain with a minimal increase in the computational cost. It is based on an analogy between flow and pressure with current and voltage in electricity. This is the reason why $C_d^k$ is called distant capacitance and $R_p^k$ and $R_d^k$ are respectively called proximal and distant resistances. These three parameters are positive real numbers.
\end{itemize}
\end{itemize}


\textcolor{felipe}{For the numerical solution of \eqref{eq:Navier_Stokes}, we use \textsl{FeLiSCe}, a finite elements method (FEM) software developed at Inria. Spatial discretization of the carotid geometry leads to a tetrahedron mesh of 42659 vertices and 225196 tetrahedra. The meshing process is in charge of MMG \cite{mmg3d}. Time is discretized with a semi-implicit backward Euler scheme with time-step $\delta t = 2 \cdot 10^{-3} s$, which means that the convective term in the Navier-Stokes equations is written with the velocity explicitly and the velocity gradient implicitly in order to circumvent the non-linearity of the system. An explicit scheme is used to numerically solve the ODE on the distal pressure in the Windkessel model. In addition, a backflow stabilization is added in order to address potential instabilities in the outlet boundaries (see, e.g., \cite{black_flow_benchmark_bertoglio}). Data visualization is done by using ParaView \cite{paraview} and Vizir \cite{vizir}.}

\textcolor{felipe}{The coupled problem for velocity and pressure is discretized using $\mathbb{P}_1-\mathbb{P}_1$ Lagrange elements, with a monolithic approach. In order to avoid the inf-sup constraint imposed by the saddle point nature of the problem we use the Brezzi-Pitkäranta stabilization approach (for further details, we refer to \cite{EG2013} and \cite{brezzi1984}), that basically perturbs the weak form of the mass conservation equation with a \textsl{stiffness-like} term that scales with the square of the elements size. Each forward simulation has $\mathcal{N} = 127977$ degrees of freedom at each time step. 
The linear systems are solved by means of a GMRES \cite{gmres} method with additive Schwarz preconditioning.}


\begin{remark}
Note that one could use more sophisticated models involving, for instance, fluid-structure interactions or more refined Winkessel models for the pressure. Our present model is a trade-off between its degree of realism and the difficulty and time to solve it. We refer to \cite{Formagia_Cardiovascular_Mathematics} for a detailed overview of cardiovascular modeling.
\end{remark}

Let us define the manifold of solutions that we consider in our numerical experiments. We set the following coefficients to a fixed value
\begin{equation}
\begin{cases}
\rho & = 1 \; \text{g}/\text{cm}^3 \\
\mu &= 0.03 \; \text{Poise} \\
C_d^k &= 1.6 \times 10^{-5} \text{ for $k=1,2$} \\
R_{p}^k &= 7501.5  \text{ for $k=1,2$} \\
p_d^k &= 1.06 \times 10^5 \text{ for $k=1,2$} \\
R_d^1 & = 60012 \\
\end{cases}
\end{equation}
The ratio of the distal resistances for the Windkessel model at the outlets is introduced:
$$
\eta \coloneqq R_{d}^1 / R_d^{2} = 60012 / R_d^{2}.
$$
This parameter plays an important role since it impacts on how the blood flow splits between the two branches. When $\eta\to 0$ or $\infty$, one branch is obstructed and the blood tends to flow through the other branch. In the following, we call this situation an arterial blockage. When $\eta\approx 1$, the flow splits more or less equally and there is no blockage.

We define the heart rate as the number of cardiac cycles per minute, that is,
$$
\text{HR} \coloneqq 60/T_c,
$$
where $T_c>0$ is the cardiac cycle duration expressed in seconds. We have $T_c = T_{sys}+T_{dia}$, where $T_{sys}$ and $T_{dia}$ are the duration of the systole and diastole respectively.

The manifold $\cM$ is generated by the variations of the following six parameters
\begin{equation}
\label{eq:params}
\begin{cases}
t & \in [0, T] \\
\text{HR} & \in [48, 120] \\
s & \in [0, 0.2]  \\
T_{sys} & \in [0.2863, 0.3182] \; \text{s.} \\
u_0 & \in [17, 20] \; \text{cm/s} \\
\eta & \in [0.05, 0.2] \cup [0.5, 1.5] \cup [5, 20]
\end{cases}
\end{equation}
Note that the time $t$ is also seen as a parameter. The parameter set is thus
$$
\rY = \{  (t, \text{HR}, s, T_{sys}, u_0, \eta) \in \bR^6 \,:\, t\in [0,T],\,\text{HR} \in [48, 120],\,s\in  [0, 0.2],\, \dots   \} \subset \bR^6
$$
and the manifold of solutions is
$$
\cM \coloneqq \{  u(y) \in [H^1(\Omega)]^3\, :\, y \in \rY \}.
$$
At this point, several comments are in order:
\begin{itemize}
\item Note that we only consider velocity fields aince in the present work we are only concerned by the reconstruction of the blood flow velocity. The reconstruction of other quantities of interest, such as the pressure, will be addressed in a forthcoming work.
\item For each $y\in\rY$, the velocity $u(y)$ is a function of $[H^1(\Omega)]^3$. In the following, we will view it as a function from
$$
V \coloneqq [L^2(\Omega)]^3,
$$
which, endowed with the inner product,
$$
\< (v_1, v_2, v_3),  (w_1, w_2, w_3)\>
\coloneqq
\sum_{i=1}^3 \< v_i, w_i \>_{L^2(\Omega)},\quad \forall (v, w) \in [L^2(\Omega)]^3,
$$
defines a Hilbert space.
\item Since the time variable has been included as a parameter, a simple way to build nonlinear reduced models is to set a window parameter $\tau>0$ and consider the subset $\cM^{(k)}=\cM_{[t_{k}-\tau, t_{k}+\tau]} \subset \cM$, where $t$ is restricted to the interval $[t_{k}-\tau, t_{k}+\tau]$ of size $2\tau$ centered around a given time $t_k$. We can then build reduced models to reconstruct this specific time interval. As we will see in the numerical experiments, this strategy is very effective in our problem because the velocity presents two regimes given by the systole and diastole periods.

\item  The computation of reduced models involves a discrete training subset $\widetilde\cM_{\train} \subset \cM_h$ which, in the experiments below, involves $\#\widetilde\cM_{\train} = 78528$ snapshots $u(y)$. The parameters are chosen from a uniform random distribution and we only save the solutions during the second cardiac cycle of each simulation.

\item For the purposes of illustrating the potential of the method for diagnoses, we introduce a  notion of sickness in terms of the arterial blockage in the following way.
\begin{definition}[Sick patient]
We say that the output of the simulation corresponds to a healthy patient when $\eta \in [0.5, 1.5]$. Outside of this range, simulations correspond to sick patients.
\end{definition}
We thus have the partition
$$
\cM = \cM_{\healthy} \cup \cM_{\sick},\qquad \cM_{\healthy} \cap \cM_{\sick} = \emptyset
$$
with $\cM_{\healthy} \coloneqq \cM_{\eta \in [0.5, 1.5]}$ and $\cM_{\sick} = \cM_{\eta \in [0.05, 0.2]\cup [5,20]}$.
\end{itemize}

\subsection{Measurements}
\label{subsec:measurements}

\textcolor{felipe}{In a regular partition of $[0,T]$, we are given Doppler ultrasound images that contain information on the blood velocity on a subdomain of the carotid. For typical ultrasound machines with acquisition time of 0.1 milliseconds, a CFI is built with, for instance, 32 consecutive B-mode frames, which leads to a measurement sampling of 32 milliseconds. From the image, the  observations $\ell_i(u)$ are extracted and used to build a complete time-dependent 3D reconstruction of the blood velocity in the whole carotid $\Omega$.}

Depending on the technology of the ultrasound device, there are two different types of velocity images. In most cases, ultrasound machines give a scalar mapping which is the projection of the velocity along the direction $b$ of the ultrasound probe. This mapping is called color flow image (CFI, see Figure \ref{fig:cfi}). In more modern devices, it is possible to get a 2D vector flow image corresponding to the projection of the velocity into the plane. This mapping is called vector flow image (VFI, see Figure \ref{fig:vfi}). For both imaging modes, the velocity is estimated by time averaging techniques (we refer to  \cite{Kasai_auto_correlation} and \cite{VFI_Jensen_Nikolov} for further details). 

\begin{figure}[!htbp]
\centering
\subfigure[Color flow image (CFI) \label{fig:cfi}]{ 
\includegraphics[height = 5 cm]{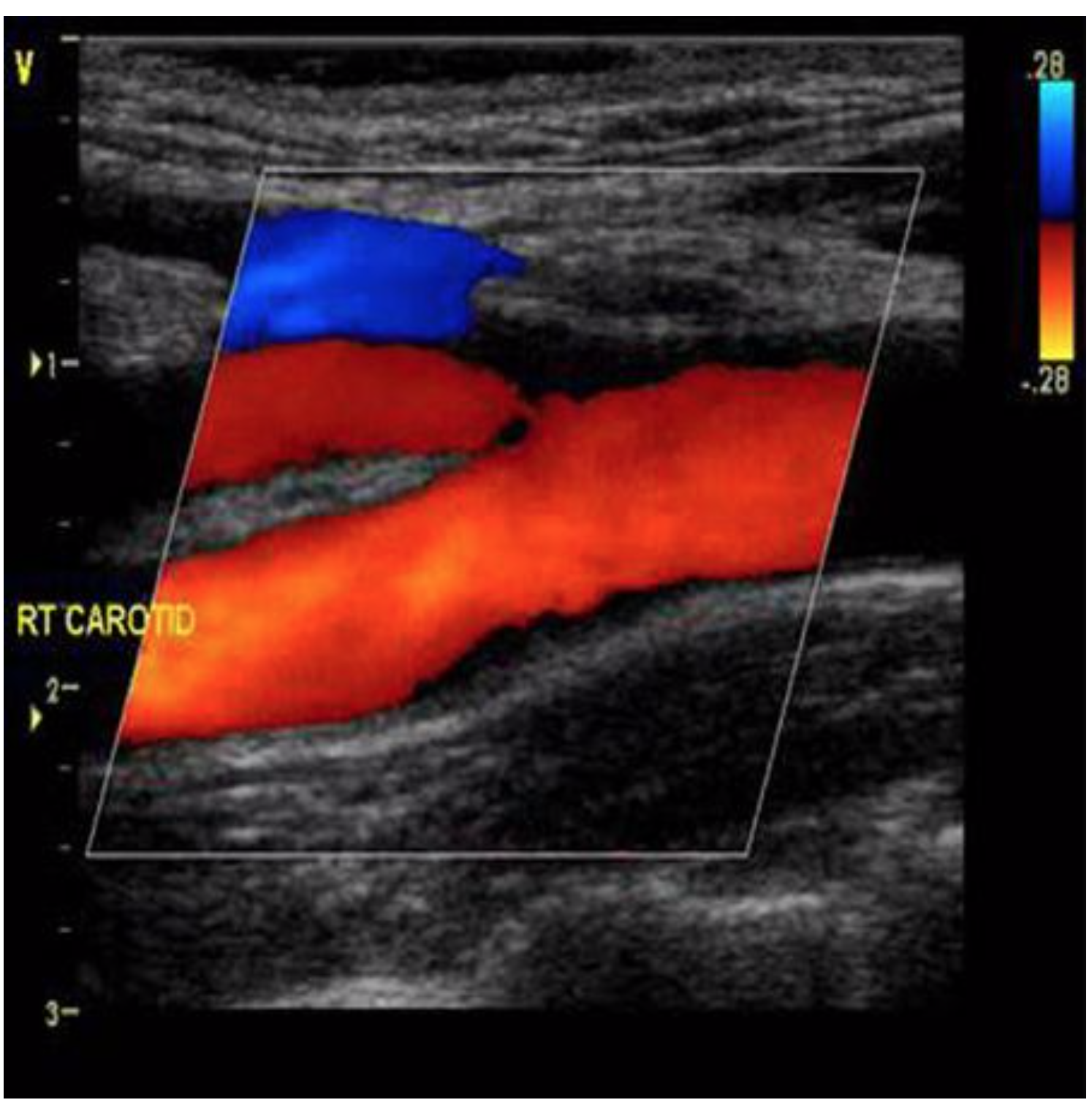}}
\subfigure[Vector flow image (VFI) \label{fig:vfi}]{ 
\includegraphics[height = 5 cm]{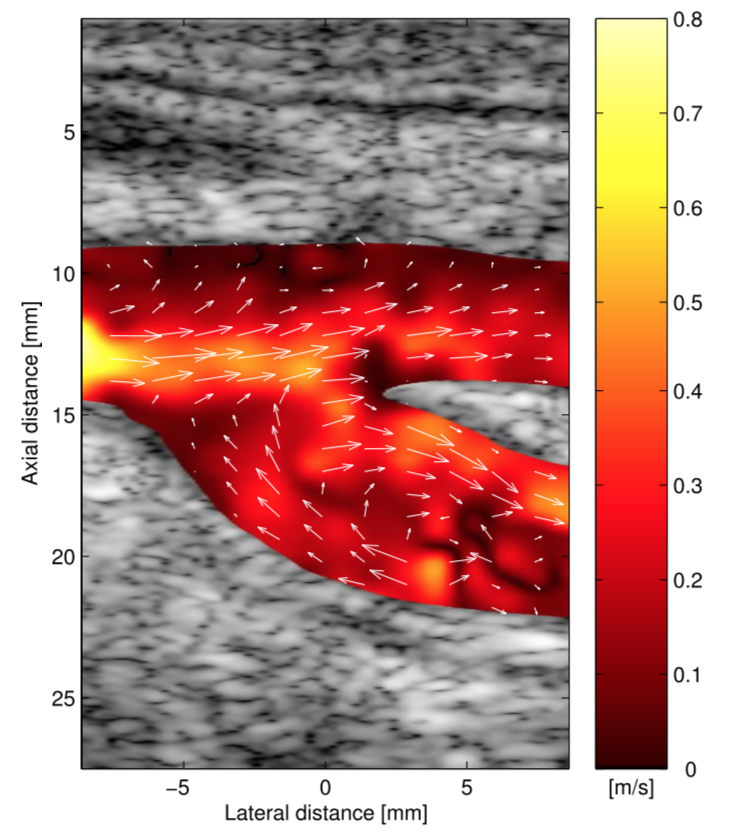}}
\caption{Velocity image of the common carotid bifurcation.}
\label{fig:CFI_and_VFI}
\end{figure}

In the following, we work with an idealized version of CFI images. For each time $t$, a given image is a local average in space of the velocity projected into the direction in which the ultrasound probe is steered. More specifically, we consider a partition of $\Omega = \cup_{i=1}^m \Omega_{i} $ into $m$ disjoint subdomains (voxels) $\Omega_{i}$. Then, from each CFI image we collect
\begin{equation}
 \ell_i(u) = \int_{\Omega_i} u \cdot b ~ d\Omega_i,\quad 1\leq i\leq m,
  \label{eq:the_measures1D}
\end{equation}
where $b\in\mathbb{R}^3$ is a unitary vector giving the direction of the ultrasound beam. From  \eqref{eq:the_measures1D}, it follows that the Riesz representers of the $\ell_i$ in $V$ are simply
$$
\omega_i =  \chi_{i} b,
$$
where $ \chi_{i} $ denotes the characteristic function of the set $\Omega_i$. Thus the measurement space is
$$
\Wm = \Wm^{(CFI)} \coloneqq \vspan\{\omega_i\}_{i=1}^m.
$$
Since the voxels $\Omega_i$ are disjoint from each other, the functions $\{\omega_i\}_{i=1}^m$ are orthogonal and therefore having a CFI image is equivalent to having

\begin{equation}
  \omega = \mathbb{P}_{W_m} u = \sum_{i=1}^m \innerp{ \omega_i} { u } \omega_i =   \sum_{i=1}^m \ell_i(u) \omega_i.
  \label{eq:the_measures_projectionWm}
\end{equation}

\begin{remark}
The case of VFI images can be treated similarly. This imaging mode gives $2m$ measurements
$$
 \ell_i(u) = \int_{\Omega_i} u \cdot b ~ d\Omega_i,\quad 1\leq i\leq m,
$$
and
$$
\ell_{m+i}(u) = \int_{\Omega_i} u \cdot b_{\perp} ~ d\Omega_i,\quad 1\leq i\leq m,
$$
where $b$ is the unitary vector giving the direction of the ultrasound beam and $b_{\perp}$ is the unitary vector perpendicular to it and contained in the image plane. Therefore, $\Wm^{(VFI)} = \Wm^{(CFI)} \oplus  \vspan\{ \chi_{\Omega_i} b_{\perp} \}_{i=1}^m$ which is a space of dimension $2m$. This clearly shows that the additional direction $b_{\perp}$ enriches the quality of the measurements in the sense that for any $u\in V$, the approximation error $\Vert u - P_{\Wm}u\Vert$ will be smaller with the VFI mode than with the CFI one. As a result, the CFI mode which we consider in our examples is a more challenging case since the measurements contain less information. 
\end{remark}

\subsection{Reconstruction on a first example with healthy patients}
\label{subsec:firstTest}

To validate our method, we first consider a simple example where we only work with healthy patients, so the manifold is $\cM_\healthy$. One color flow image contains information of the velocity averaged over 552 voxels. As explained in  section \ref{subsec:measurements}, this sets the dimension of the observations space  to $m=552$ (whereas for VFI, $m=1104$). As a result, each image can be seen as an observation $\omega\in\Wm$ (see Figure \ref{fig:medidas_CFI}). The dimension $m=552$ may seem quite large but it is representative of the one provided by modern pulsated ultrasound devices. 297 healthy patients are simulated in order to build the training set. For each one of them a number of snapshots containing the second cardiac cycle is stored so there are multiple snaptshots per patient due to time marching. 
This leads to a training set composed of $\#\widetilde\cM_{\train} = 56383$ snapshots. The performance of the algorithms is tested on a test set $\widetilde\cM_{\test}$ of 32 healthy patients each one with a parameter configuration in range but different from those inside $\widetilde\cM_{\train}$.

\begin{figure}[!htbp]
  \centering
  \includegraphics[height = 5 cm]{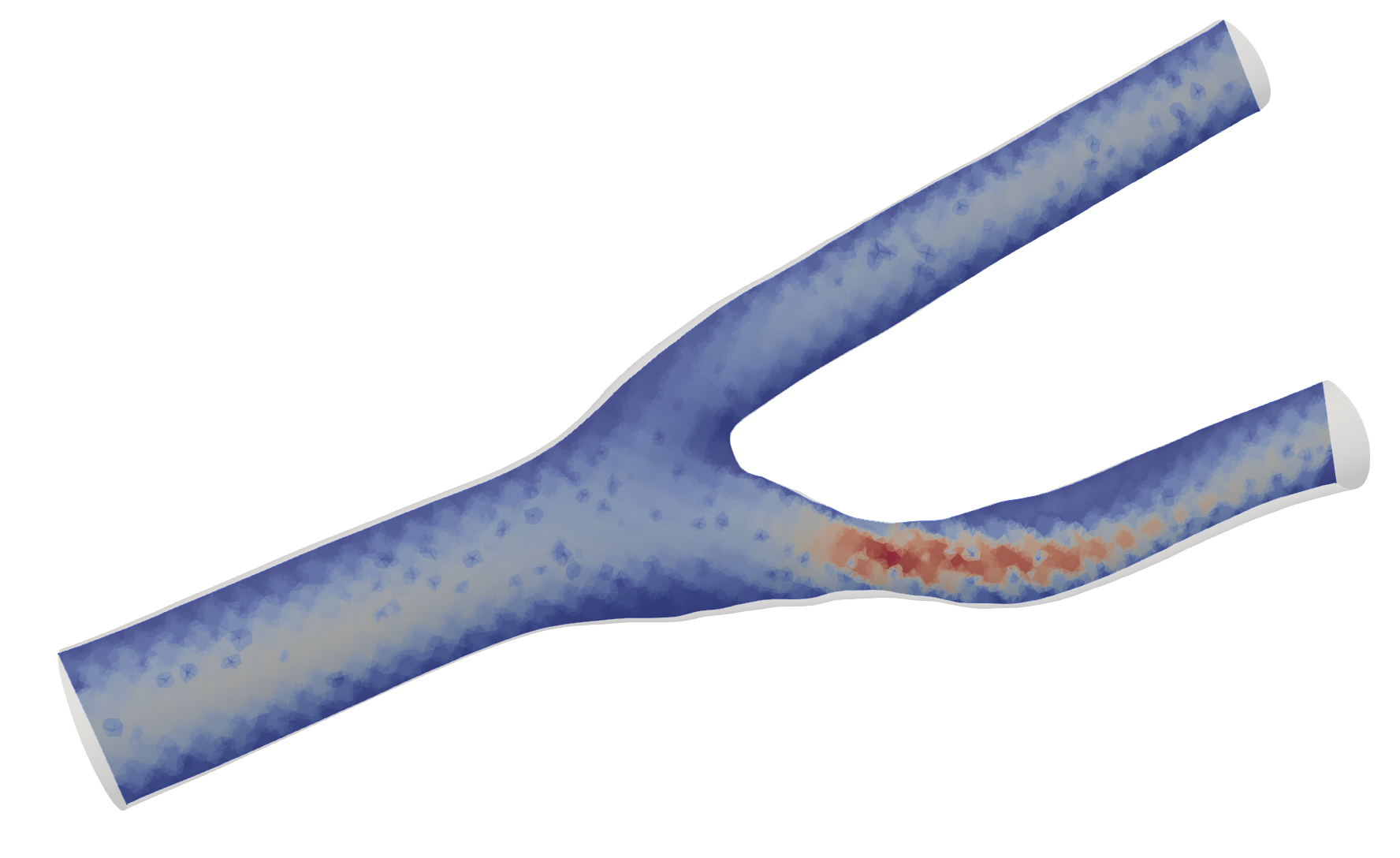} \\
  \includegraphics[height = 1 cm]{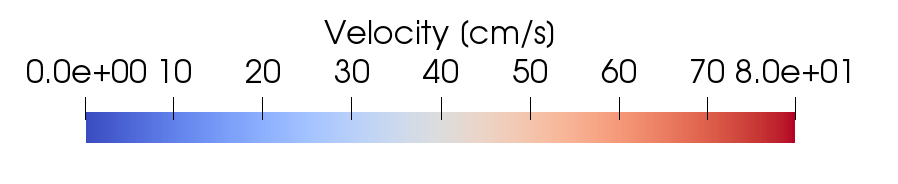} \\
  \caption{Example of synthetic CFI measures used in first example. The image leads to an space $W_m$ of dimension $m=552$}
  \label{fig:medidas_CFI}
\end{figure}

For the comparison, we implement the algorithms outlined in section \ref{sec:summary-methods}. We recall them below and give specific details related to the current test case:
\begin{enumerate}
\item \emph{Linear PBDW:} The space $V_n$ is built from a classical POD of the whole training set $\widetilde\cM_{\train}$.
\item \emph{Nonlinear algorithm with manifold partitioning:} In real medical examinations, the heart rate $\text{HR}$ of the patient and the time $t$ in which the ultrasound image is taken are known. We exploit this fact to decompose $\widetilde\cM_{\train}$ into $K=IJ$ subsets
\begin{equation}
\widetilde\cM_{\train} = \bigcup_{ \substack{ k = (i, j)  \\ (i,j) \in \{ 1,\dots, I \} \times \{ 1,\dots, J \}} }  \widetilde \cM^{(k)},
\end{equation}
where, for each $k = (i, j)$ in $\{ 1,\dots, I \} \times \{ 1,\dots, J \}$,
\begin{equation}
\widetilde \cM^{(k)}= \{  u\in \widetilde\cM_{\train} :  t \in [t_i - \tau, t_i +
  \tau],~\text{HR} \in [\text{HR}_j - \delta_{\text{HR}}, \text{HR}_j +
  \delta_{\text{HR}} ] \}.
  \label{eq:training_manifold}
\end{equation}

Among the possibilities to build the manifold partitioning, our approach is driven by the reconstruction error of the elements inside $\widetilde\cM_{\test}$. That is to say, we have computed the reconstruction $u^*_i$ of $u_i$ for $i=1,\ldots, \# \widetilde\cM_{\test}$ for a set of couples $\( \delta_{\text{HR}}, \tau \)$ with a fixed dimension for the reduced model ($n=30$). Hence, we can pick the optimal values in this sense:
$$
\( \delta_{\text{HR}}^*, \tau^* \) = \argmin_{\delta_{\text{HR}},\tau}\( \frac{1}{\# \widetilde\cM_{\test}} \sum_{i=1}^{\# \widetilde\cM_{\test}} \frac{ \norm{u_i - u^*_i \( \delta_{\text{HR}}, \tau \)}}{\norm{u_i} } \)
$$
which has led in our context to the sizes $\tau = T_c /10$ and $\delta_{\text{HR}} = 5 $ beats per minute.

For each subset  $\widetilde \cM^{(k)}$, a reduced space is built. Two constructions have been tested:
\begin{itemize}
\item POD of $\widetilde \cM^{(k)}$. We call this approach \ppa, Partitioned-POD-affine, in the plots.
\item A greedy algorithm. We call this approach \pga, Partition-Greedy-affine, in the plots
\end{itemize}
During the online reconstruction, given $t$ and HR, we select the appropriate subset $\widetilde \cM^{(k)}$ that includes $t$ and HR and reconstruct with a linear PBDW with the reduced model corresponding to $\widetilde \cM^{(k)}$.
\item \emph{Data-driven nonlinear algorithm with manifold partitioning:} This approach is labeled in the plots as \pdba, Partitioned-Data-Based-affine. Since each CFI image can be seen as an observation $\omega\in\Wm$, we run the Orthogonal Matching Pursuit algorithm of section \ref{sec:nonlinear} to build $V_n(\omega)$ and do the reconstruction. Note that the greedy search has to be done online since we need the knowledge of the measurement. 
\end{enumerate}

For each state $u\in \widetilde\cM_{\test}$, we compute the relative error: 
\begin{equation}
  e(u, A_{n,m}) = \frac{ \norm{ u - A_{m,n}(P_{\Wm}u) } }{ { \norm{u} } }
  \label{eq:error_sim_def}
\end{equation}
where $A_{m,n}(P_{\Wm}u)$ denotes any of the above four reconstruction algorithms. 
A patient in the training set is represented by a sequence of simulated states during the second cardiac cycle. This raises the interest in evaluating the reconstruction quality by looking at the following relative error in time:
\begin{equation}
  e(t, u, A_{m,n}) = \frac{ \norm{u(t) - A_{m,n}(P_{\Wm}u(t))} } {\left( \int_{T_c}^{2T_c} \norm{u}^2
      \dt\right)^{1/2}}
  \label{eq:errorl2_2}
\end{equation}
where we normalize by the total energy in the cardiac cycle $\left(\int_{T_c}^{2T_c} \norm{u(t)}^2 \dt \right)^{1/2}$.


\textcolor{felipe}{Figures \ref{fig:error_CFI_whole_sim_av} and \ref{fig:error_CFI_whole_sim_wc} give the average and worst case performance of the four methods,} 
$$
e_{\av}(A_{n,m}) = \frac{1}{\#  \widetilde\cM_{\test}}
\sum_{u\in\widetilde\cM_{\test} } e(u, A_{n,m}),
\quad 
e_{\wc}(A_{n,m}) =
\max_{u\in\widetilde\cM_{\test} } e(u, A_{n,m}).
$$

\begin{figure}
  \centering
  \subfigure[Average error]{
  \label{fig:error_CFI_whole_sim_av}
  \includegraphics[height = 5 cm]{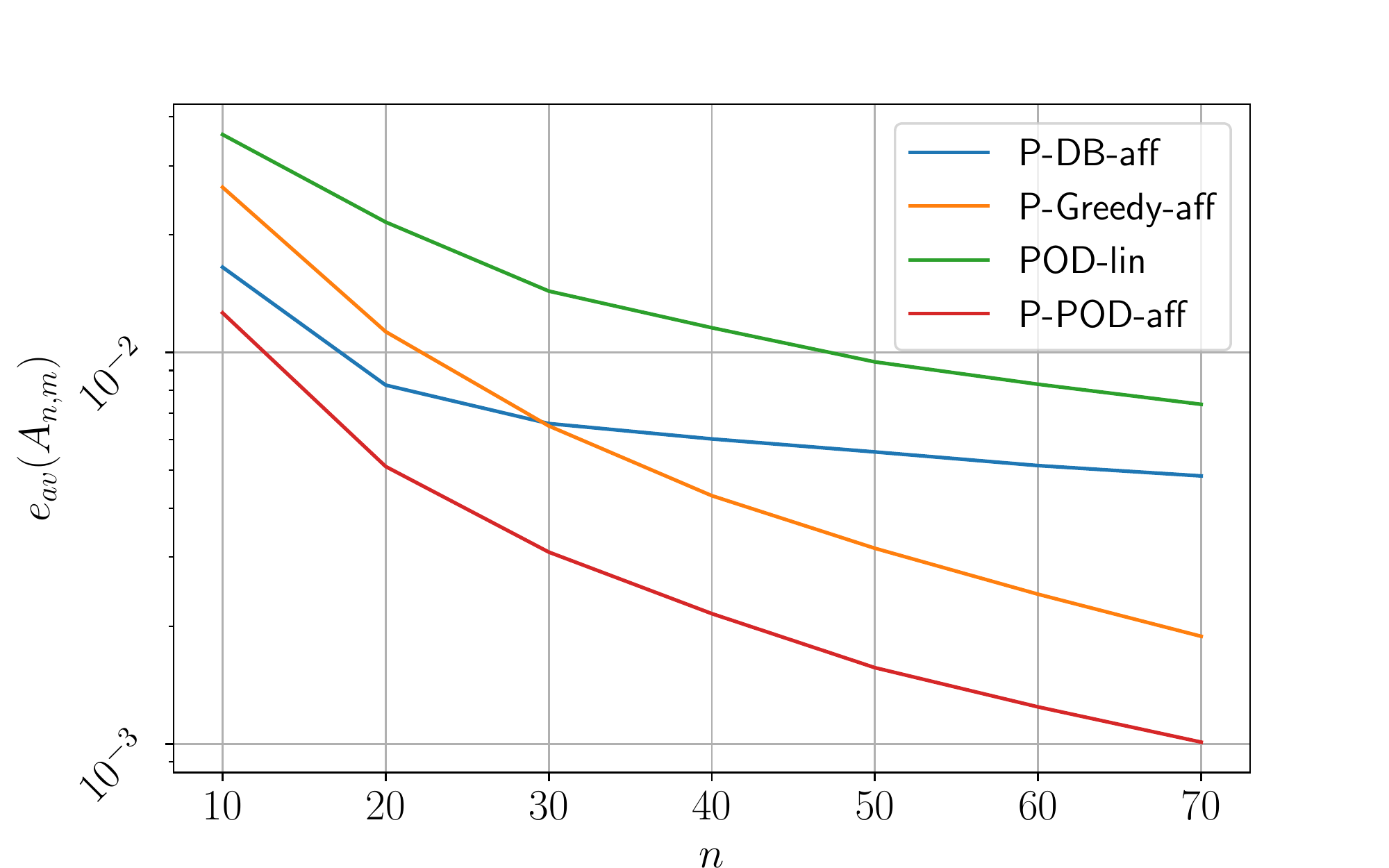}}
  \subfigure[Worst error]{
  \label{fig:error_CFI_whole_sim_wc}
  \includegraphics[height = 5 cm]{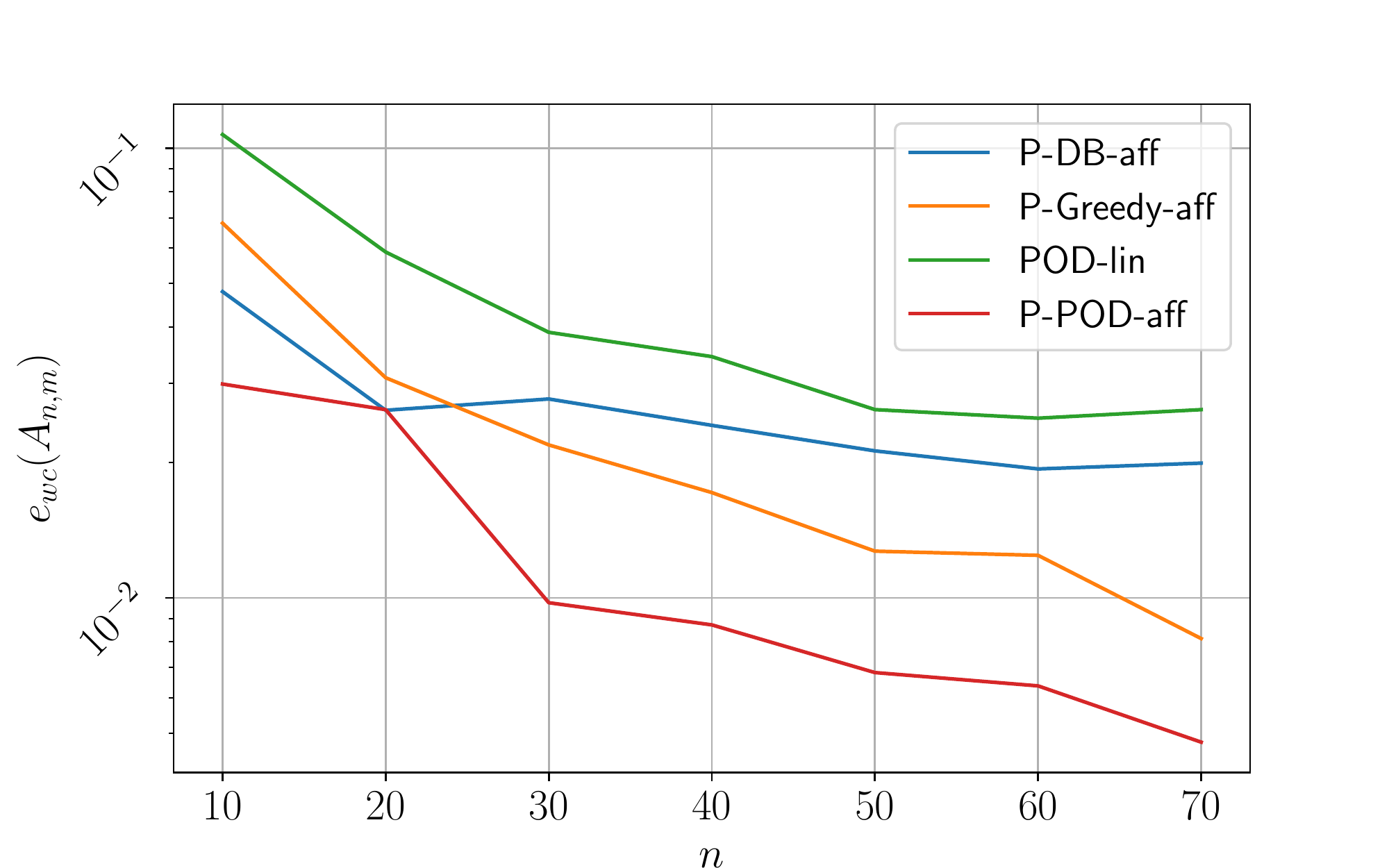}}
  \caption{Benchmark for first numerical example. The accuracy is evaluated using the error \eqref{eq:error_sim_def}}
  \label{fig:eror_CFI_mean_and_max}
\end{figure}

Note that as $n$ increases, the error decreases for all methods, except perhaps for the \pdba~ approach where the error tends to stagnate for large values of $n$. This could be due to the fact that \pdba~heavily relies on the measurement information, which, in the present application, might not deliver enough to learn reduced models $V_n(\omega)$ that improve the accuracy as $n$ grows. For instance, the Doppler measurements are close to zero in the diastole of the cardiac cycle so the information that they provide may be insufficient to build a good $V_n(\omega)$. We also see from the figure that the nonlinear method \ppa~outperforms the rest in the sense that it delivers a given target accuracy with a smaller dimension $n$ of the reduced model. For instance, if we fix a target accuracy on the average performance to $10^{-2}$, we see that the \pl~requires 40 modes to achieve it,  \pga~requires 20, \pdba~requires 17 and \ppa~requires only 10 (see Figure \ref{fig:error_CFI_whole_sim_av}).


We next fix $n=30$ and study the error in time $e(t, u, A_{m,30})$ on Figure \ref{fig:error_CFI_time}. We observe that the reconstruction tends to be better during the late diastole phase of the cardiac cycle.

\begin{figure}
  \centering
   \includegraphics[height = 5 cm]{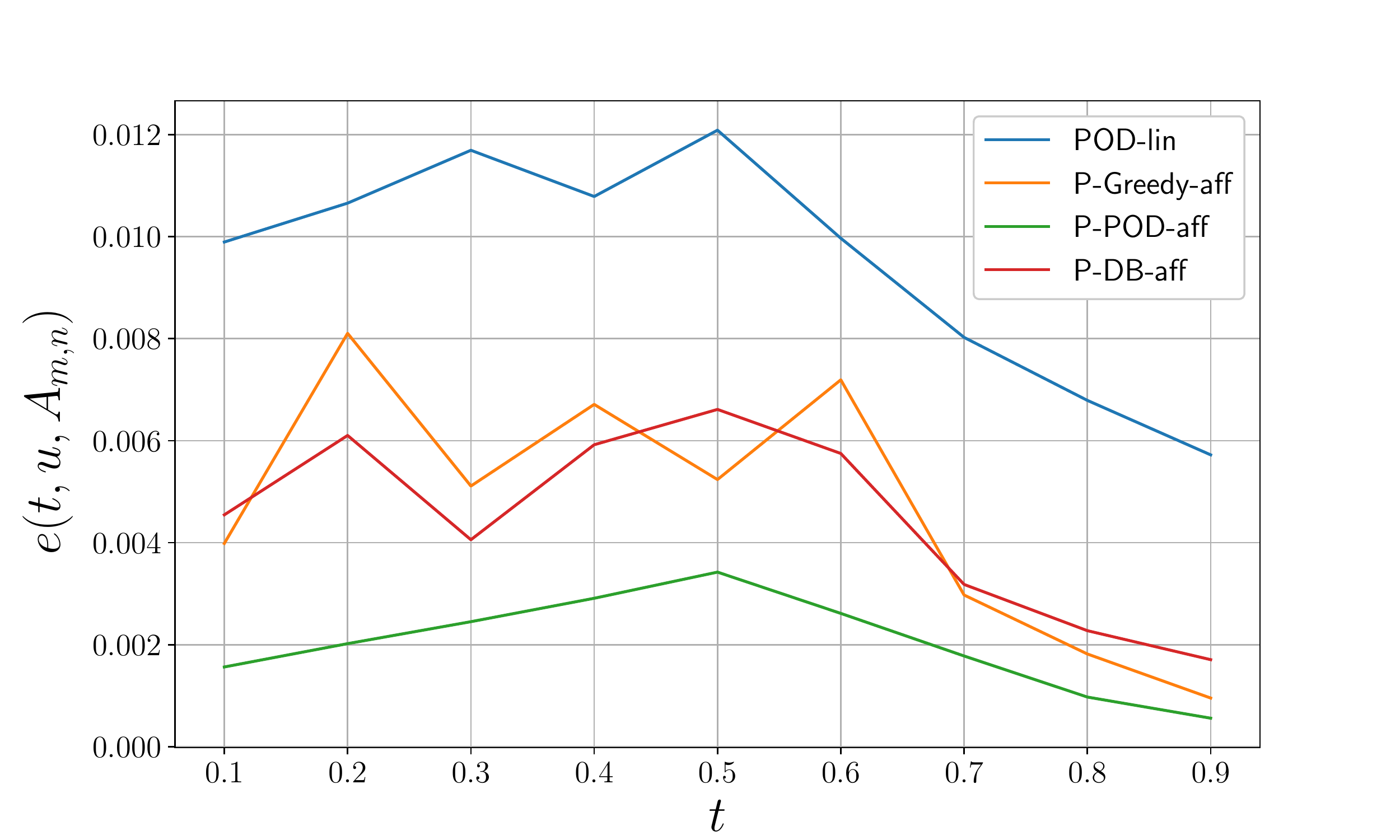}
  \caption{Average error comparison for test case described in section \ref{subsec:firstTest}. The benchmark shows the temporal evolution of the quantity \eqref{eq:errorl2_2} during the cardiac cycle. The dimension of $V_n$ is set to 30.}
  \label{fig:error_CFI_time}
\end{figure}

As discussed in section \ref{sec:algos}, the inf-sup constant $\beta(\Vn, \Wm)$ might yield to stability issues when $n\to m$ since its value tends to zero (see equations \eqref{eq:err-wc-pbdw} and \eqref{eq:err-ms-pbdw}). Figures \ref{fig:error_beta_systole} and \ref{fig:error_beta_diastole} show its behavior for the four methods during the systole and diastole period. We observe that the four methods perform similarly in terms of stability for the peak systole reconstruction. For the diastole phase, we observe that the inf-sup constant in \pdba~performs slightly worse than the rest. We think that this could be due to the fact that the measurement space $\Wm$ is not rich enough to allow \pdba~to properly learn reduced models when $n$ becomes large.

\begin{figure}
  \centering
   \includegraphics[height = 5 cm]{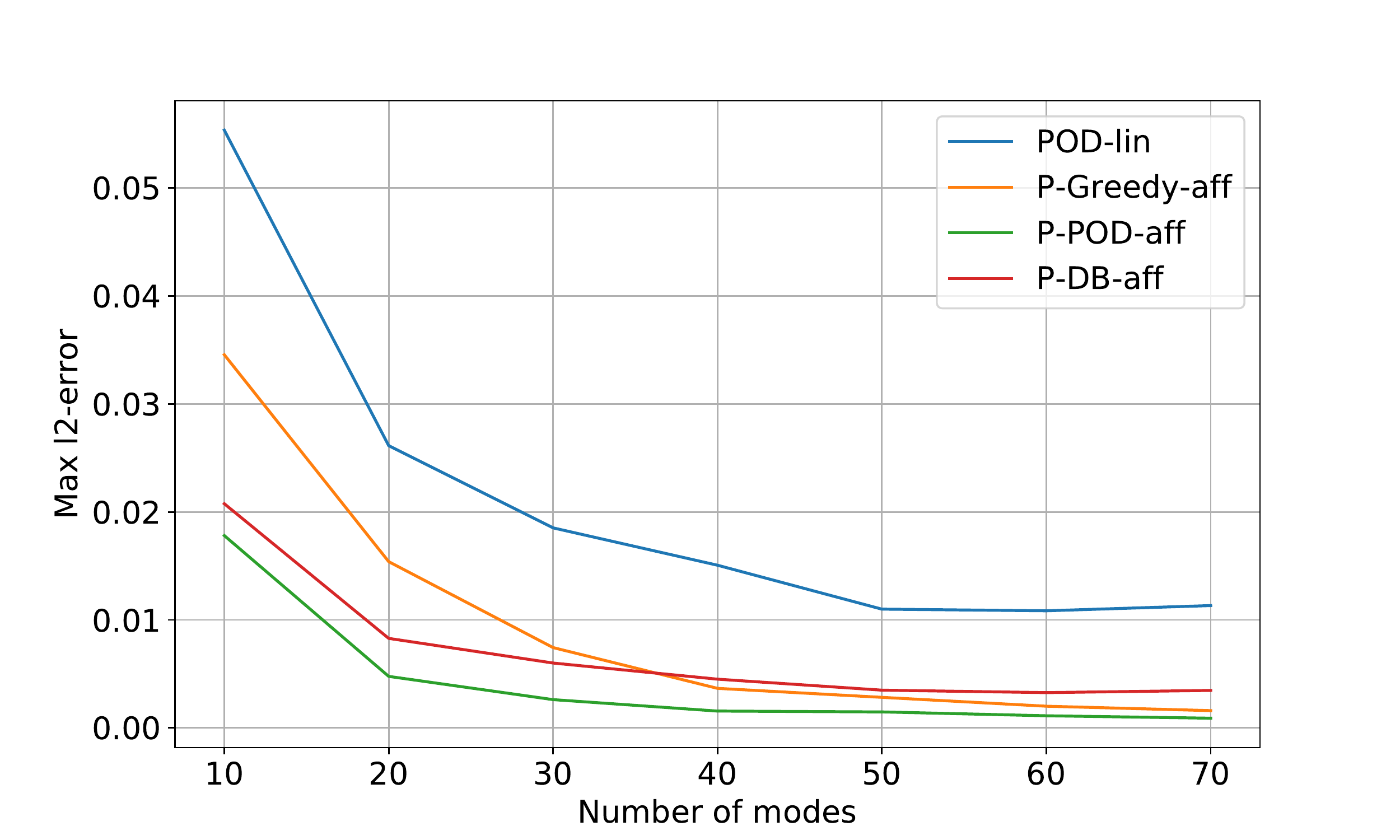}
   \includegraphics[height = 5 cm]{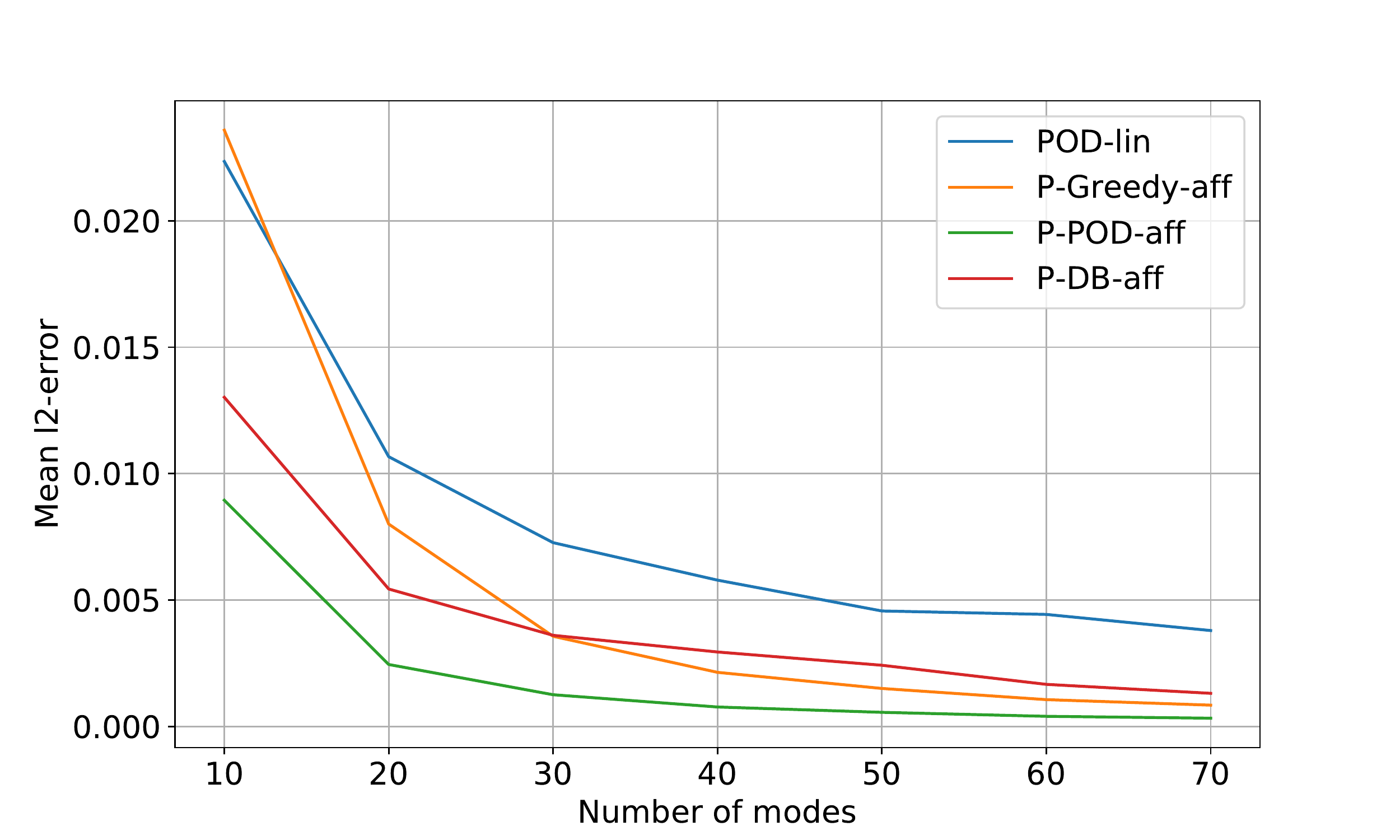}
  \caption{Comparison of reconstruction error for peak systole. Max error in the left and mean error in the right.}
  \label{fig:error_mean_max_peak}
\end{figure}

\begin{figure}
  \centering
  \subfigure[Peak systole]{ 
  \label{fig:error_beta_systole}
  \includegraphics[height = 5 cm]{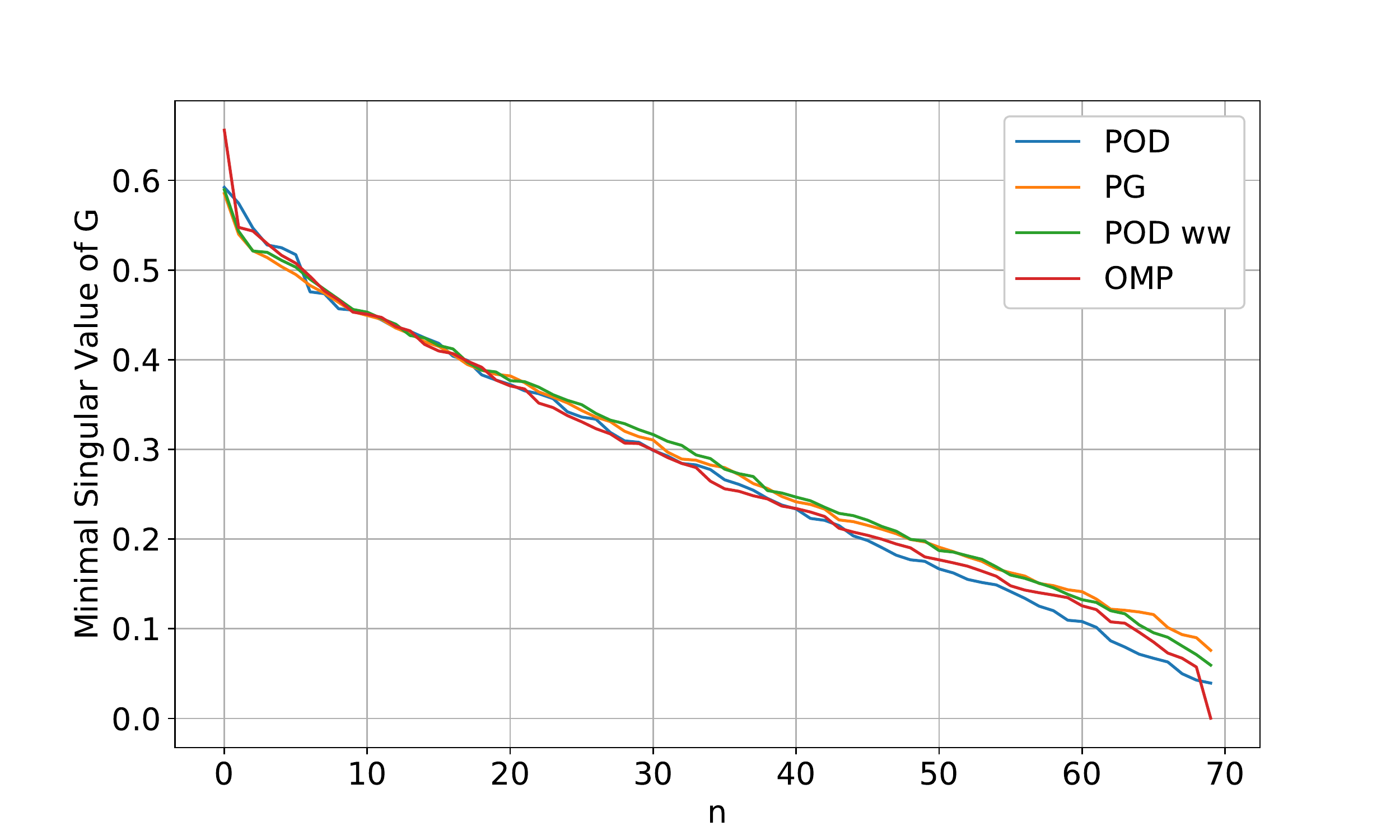}}
  \subfigure[Late diastole]{
  \label{fig:error_beta_diastole}
  \includegraphics[height = 5 cm]{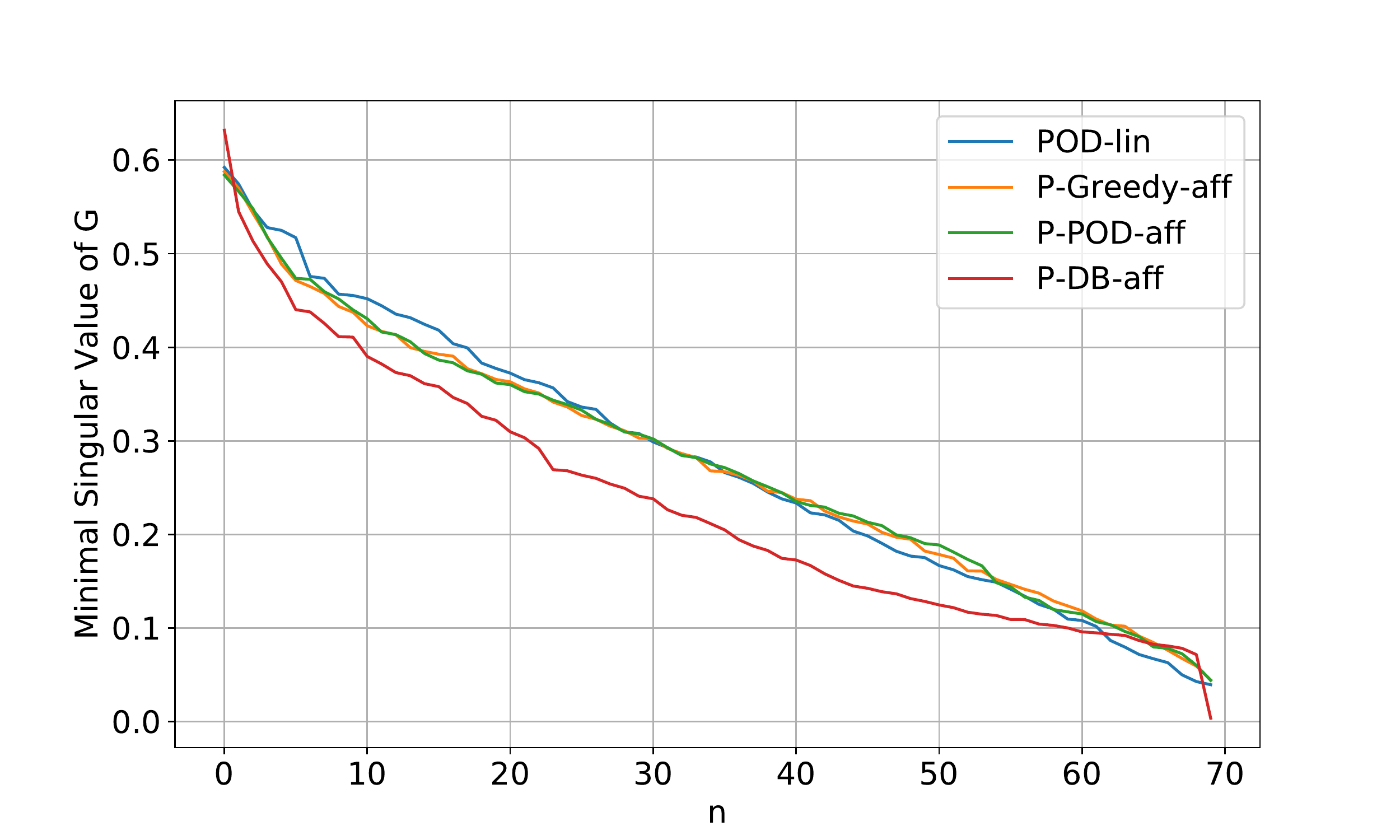}}
  \caption{Behavior of $\beta(\Vn, \Wm)$ as a function of $n$.}
  \label{fig:error_beta}
\end{figure}

\subsection{Application to arterial blockage detection}
\label{subsec:secondTest}
In this example we illustrate that even when the Doppler images do not give information on the whole carotid, we can nevertheless reconstruct the velocity field in the whole domain with  our methodology. This is important for actual practice since doctors do not have images in the whole carotid due to morphological constraints. We also show in our example that the method has potential to efficiently estimate in real time relevant quantities of interest.

We illustrate these ideas in the following example: we consider the same setting as before but now the Doppler image does not provide information about the flow in the carotid bifurcations. Therefore, the image does not \textsl{see} the flow split in the common carotid downstream (see Figure \ref{fig:sint_blockage_CFI}). In this example we have tested the impact of working with CFI ($m=233$) or VFI images ($m=466$).


We train our reconstruction methods on a set $\widetilde \cM_{\train}$ containing sick and healthy patients. Here, we only work with two of our previous nonlinear algorithms: 
\begin{enumerate}
\item \ppa
\item \pdba
\end{enumerate}
Figure \ref{fig:error_blockage_CFI_max_and_avg} shows the average and worst case errors 
$$
e_{\av}(A_{n,m}) =
\sum_{u\in\widetilde\cM_{\test} } e(u, A_{n,m}),
\quad 
e_{\wc}(A_{n,m}) =
\max_{u\in\widetilde\cM_{\test} } e(u, A_{n,m}),
$$
as a function of the dimension $n$ of $V_n$. Like in the previous example, both methods are delivering a very satisfactory accuracy: the average error is below $5 \cdot 10^{-2}$ for both methods for all values of $n$. The method \ppa~consisting in a partition of the manifold outperforms the data-based one, \pdba.

We next show that the method is efficient to assist in the detection of arterial blockages that may cause severe health problems like a stroke. Since a blockage alters the distribution of the velocity field after the bifurcation, a quantity of interest that could serve as a clinical index is the ratio
\begin{equation}
  r = \frac{Q_2(t_{peak})}{Q_1(t_{peak})}
  \label{eq:flow_ratio}
\end{equation}
where
$$
Q_i(t) \coloneqq \int_{\Gamma_o^i} u(t) \cdot n
$$
is the blood flow at the outlet $\Gamma_o^i$, $i=1,\,2$, and $t_{peak}$ is the peak systole instant. Figure \ref{fig:flow_reconstruction} shows the evolution of $Q_i(t)$ in time for a sick patient and its approximation with our two reconstruction methods.  We observe that, regardless of the image format (CFI or VFI), both methods deliver very satisfactory predictions of the flow.

In Figure \ref{fig:flow_ratio}, we compare the value of the exact ratio $r$ with the reconstructed one for sick and healthy patients $u\in \widetilde \cM_{\test}$. 

To define a threshold ratio $r^*$ to decide whether the patient has arterial blockage or not, we can take the average of the flow ratios between the healthier of the patients in the sick group and the sicker of the patients in the healthy group, namely,
\begin{equation}
  r^* \coloneqq  \frac{ \min_{v \in \widetilde\cM_{\sick}}  r(v) + \max_{u \in \widetilde\cM_{\healthy} } r(u) }{2}
  \label{eq:index_threshold}
\end{equation}
where $r(u)$ denotes the flow ratio associated to the velocity field $u$, as defined in \eqref{eq:flow_ratio}. In our data-base, we obtain, $r^*=1.25$, so any patient for which $r > 1.25$ will be considered as presenting high blockage risk. Note that the approximation is very close to the real value for moderate values of $r$ regardless of the image modality. However, we tend to overestimate the value for $r>1.7$. In presence of a blockage, $r$ becomes significantly larger than one so the overestimation is by far more preferable than an underestimation. Indeed, the overestimation makes our method conservative and, in the worst case, we will conclude with a false positive. However, the method will not lead to a false negative diagnosis, which would leave a sick patient without treatment/surgery. 

\begin{figure}
  \centering
  \includegraphics[height = 5 cm]{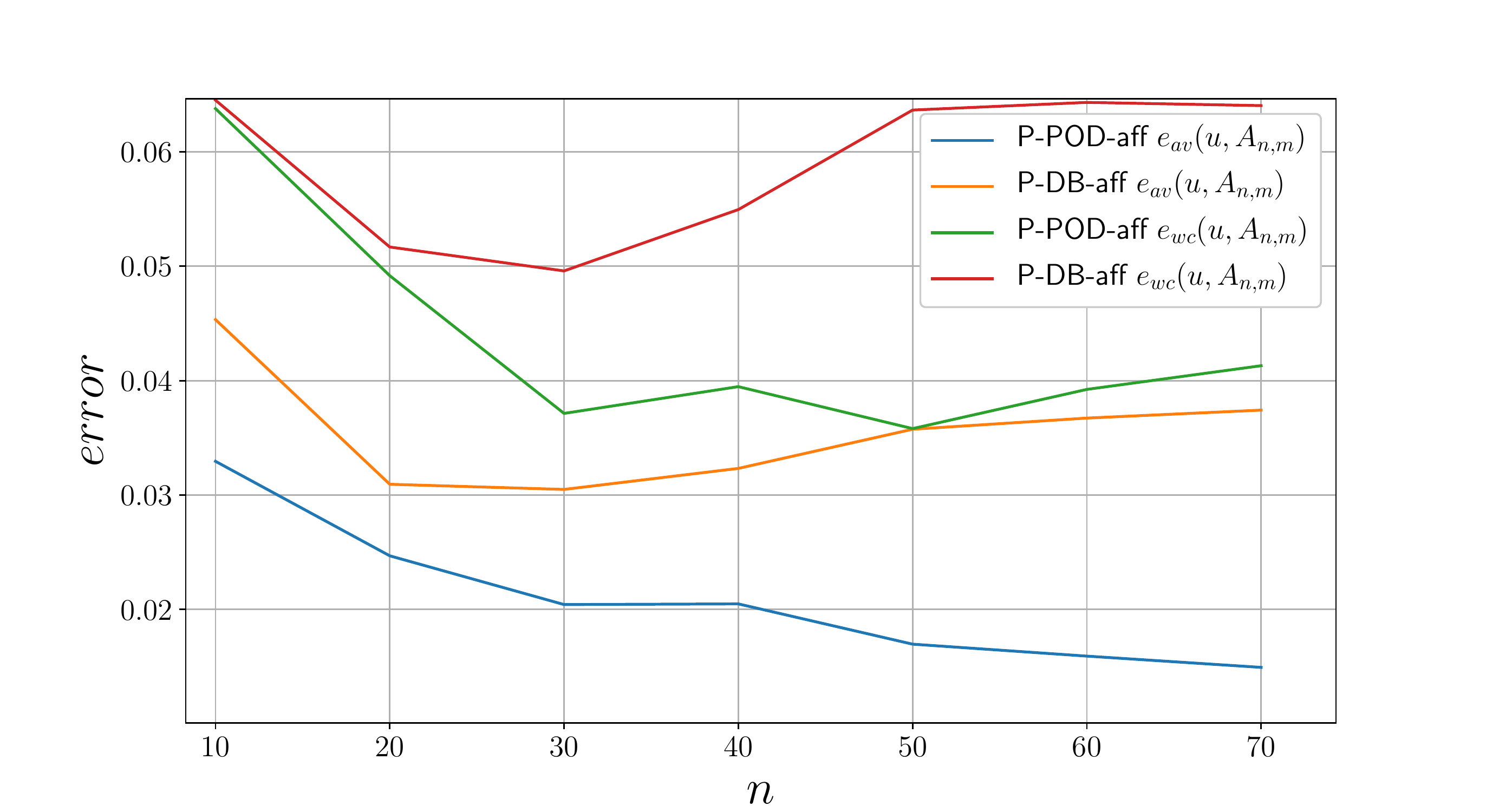}
  \caption{Average and worst reconstruction errors as a function of $n$.}
  \label{fig:error_blockage_CFI_max_and_avg}
\end{figure}

\begin{figure}
  \centering
  \subfigure[CFI]{ 
    \includegraphics[height = 5 cm]{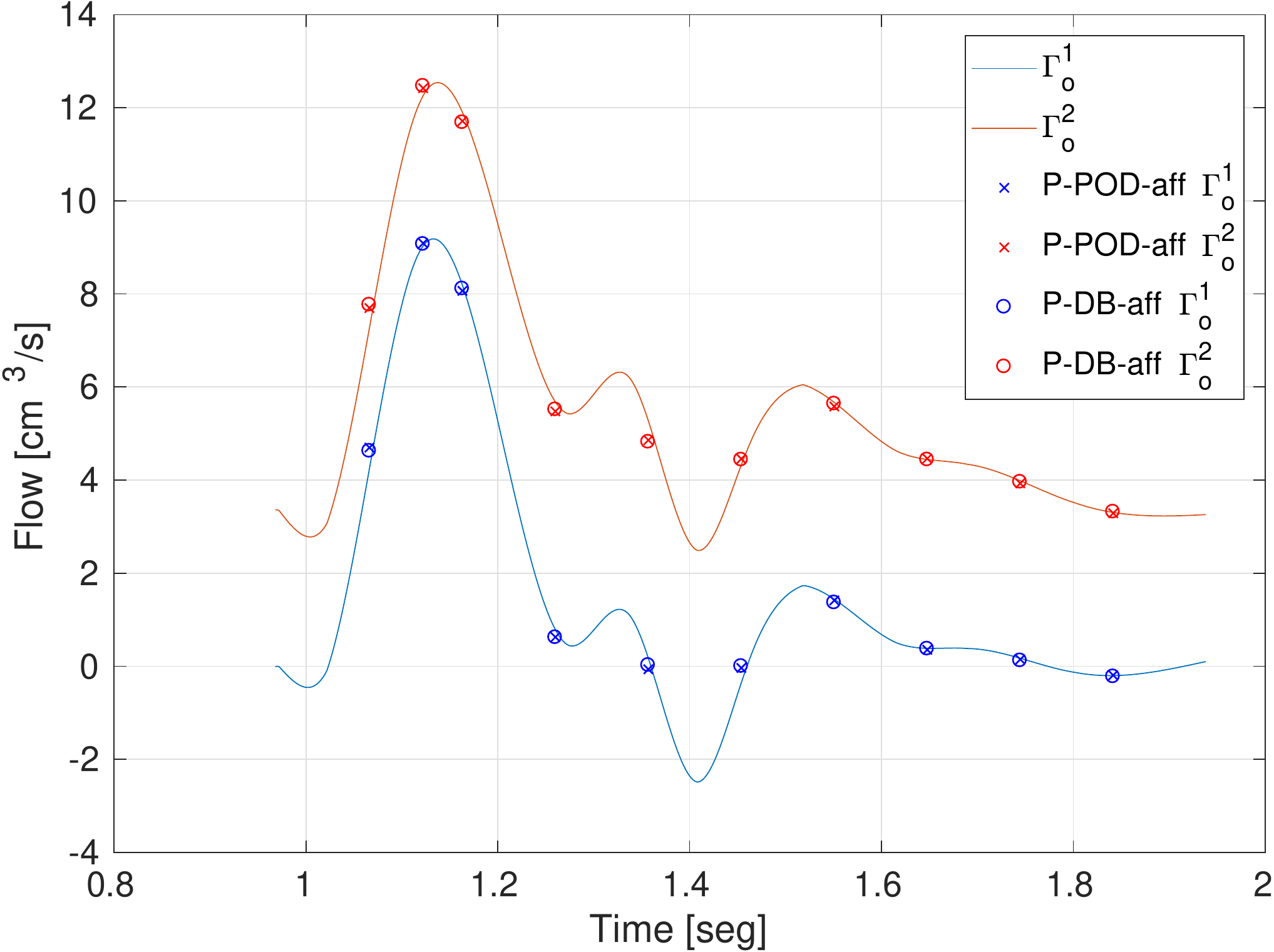}}
   \subfigure[VFI]{ 
    \includegraphics[height = 5 cm]{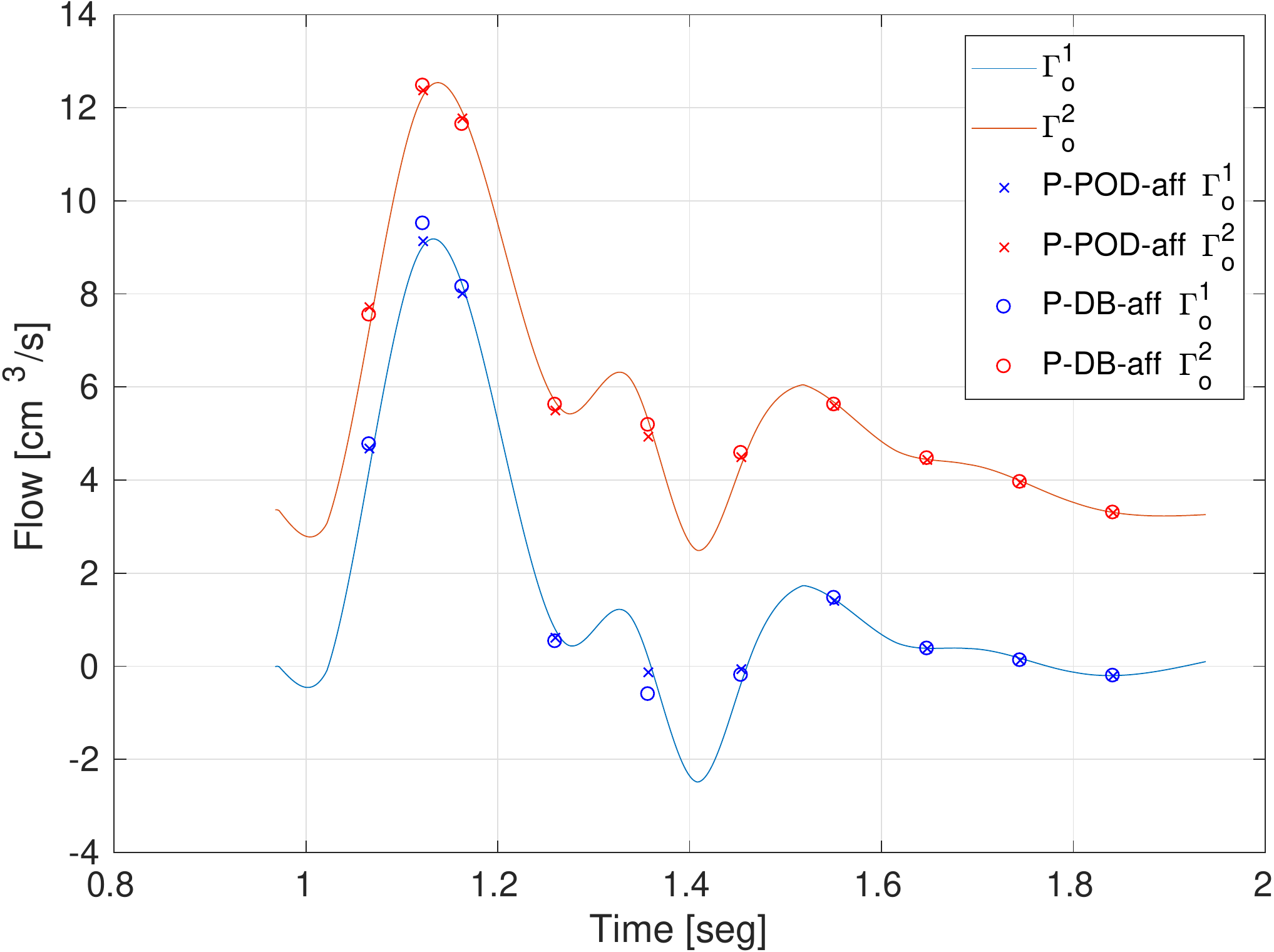}}
    \caption{Comparison of simulated and reconstructed flow at the outlets of the domain.}
    \label{fig:flow_reconstruction}
\end{figure}

\begin{figure}
  \centering
  \subfigure[CFI]{ 
    \includegraphics[height = 5 cm]{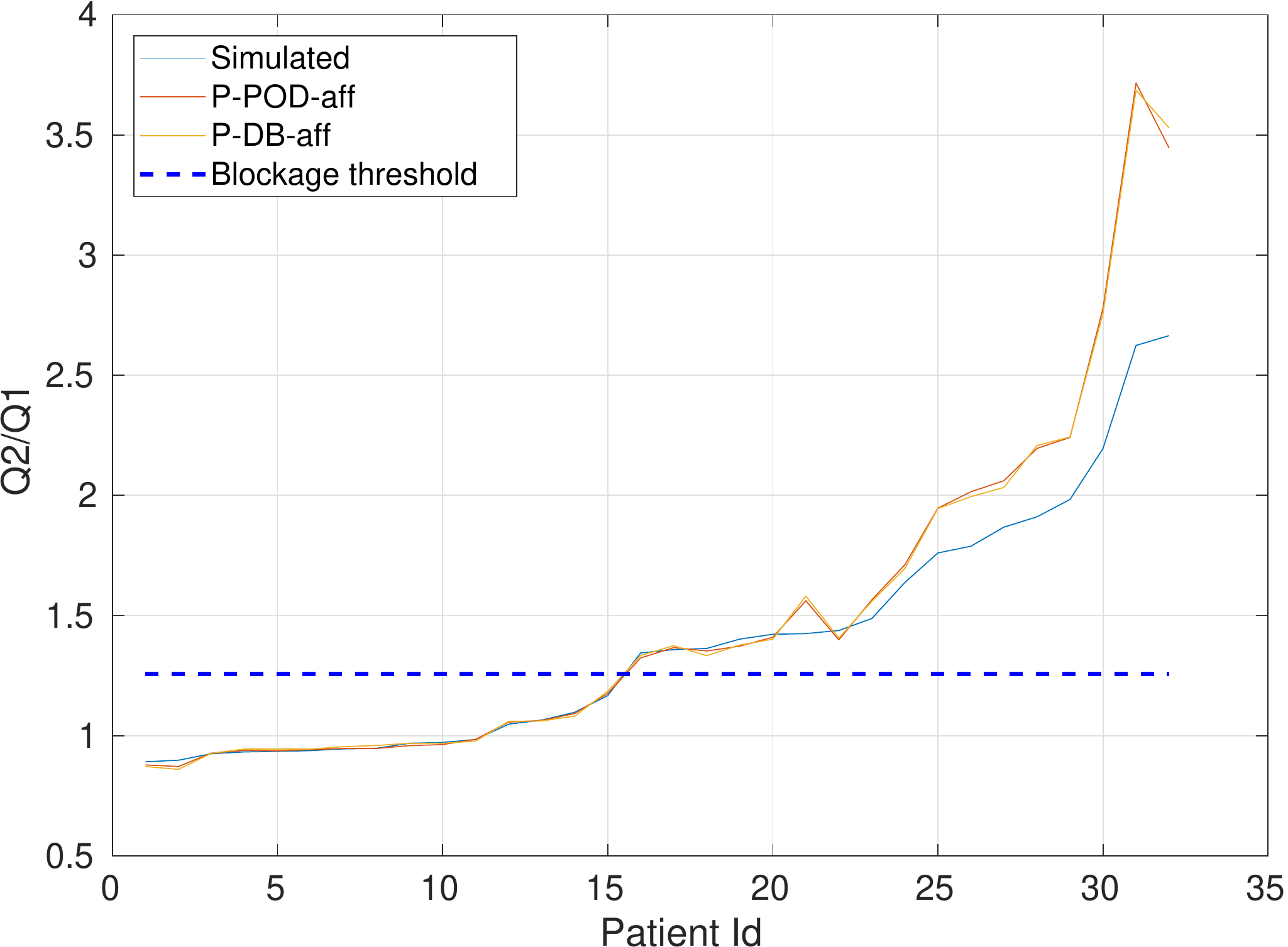}}
   \subfigure[VFI]{ 
    \includegraphics[height = 5 cm]{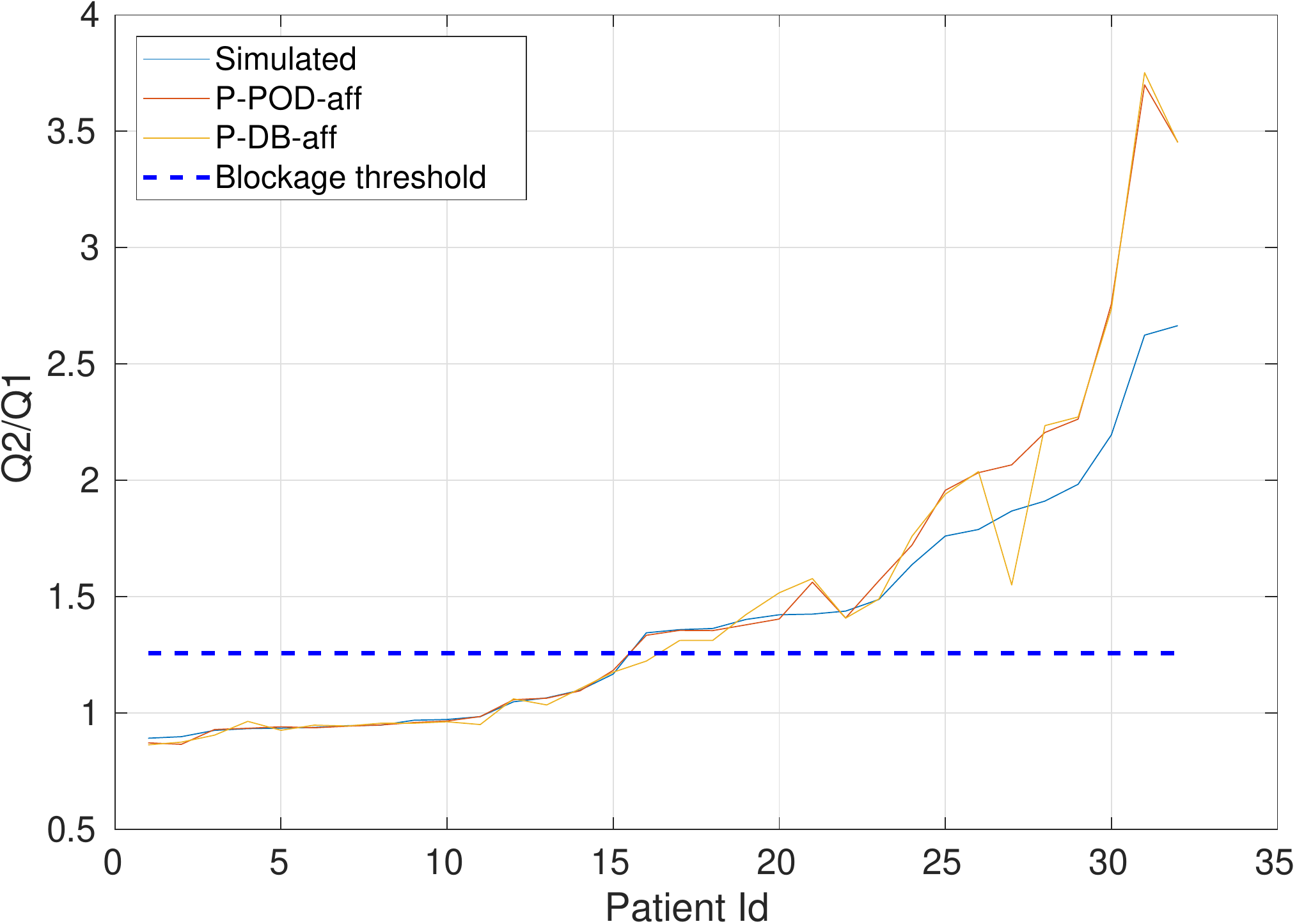}}
    \caption{Flow ratio reconstructed for blockage detection. In this test the number of modes is fixed to 30.}
    \label{fig:flow_ratio}
\end{figure}


\textcolor{felipe}{
To illustrate the reconstruction for both CFI and VFI, let us consider the forward simulation of a peak systole blood flow as depicted in Figure \ref{fig:block_GT}. Assuming this field as a ground truth, we can see the reconstruction with \ppa~and \pdba in Figures \ref{fig:blockage-image-CFI} and \ref{fig:blockage-image-VFI}.
}

\begin{figure}
  \centering
  \includegraphics[height = 7 cm, angle=90]{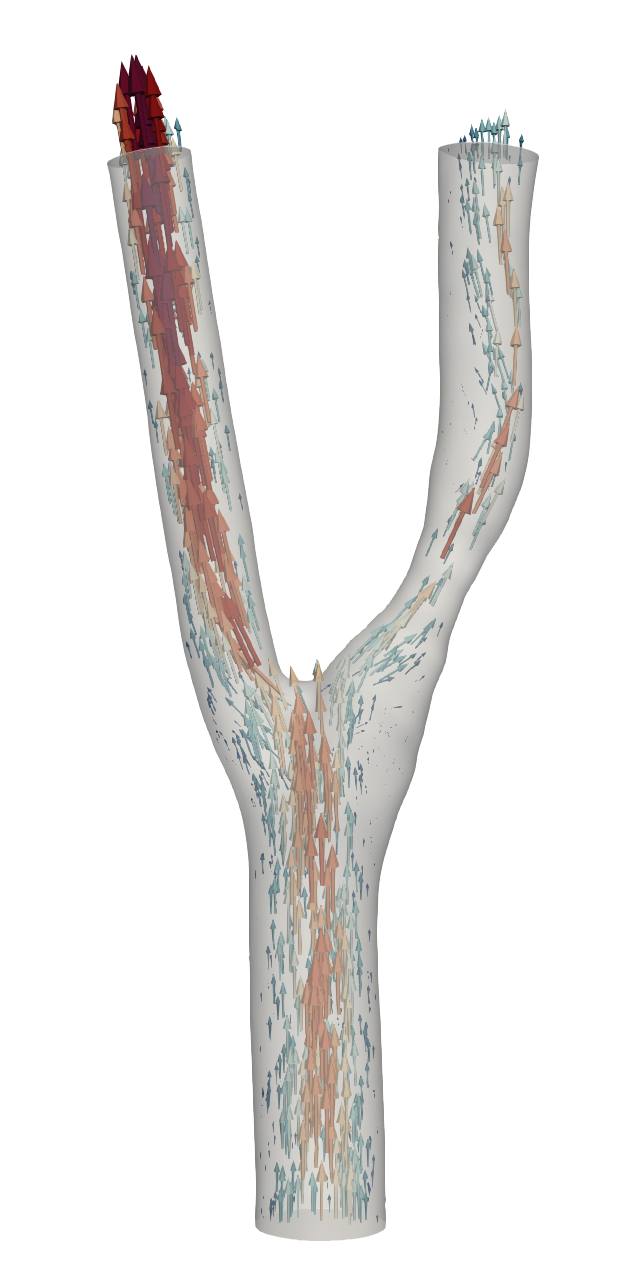} \\
  \includegraphics[height = 1 cm]{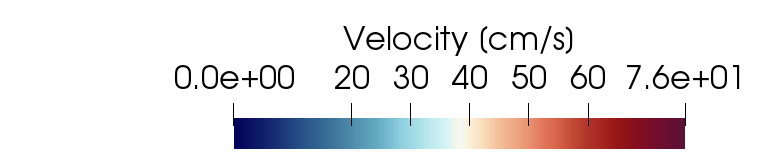}
  \caption{\textcolor{felipe}{Forward simulation and ground truth for one of our test-cases. This field belongs to the data base of patients with blockage.}}
  \label{fig:block_GT}
\end{figure}

\begin{figure}
	\centering
    \subfigure[$\omega$]{\includegraphics[height = 7 cm]{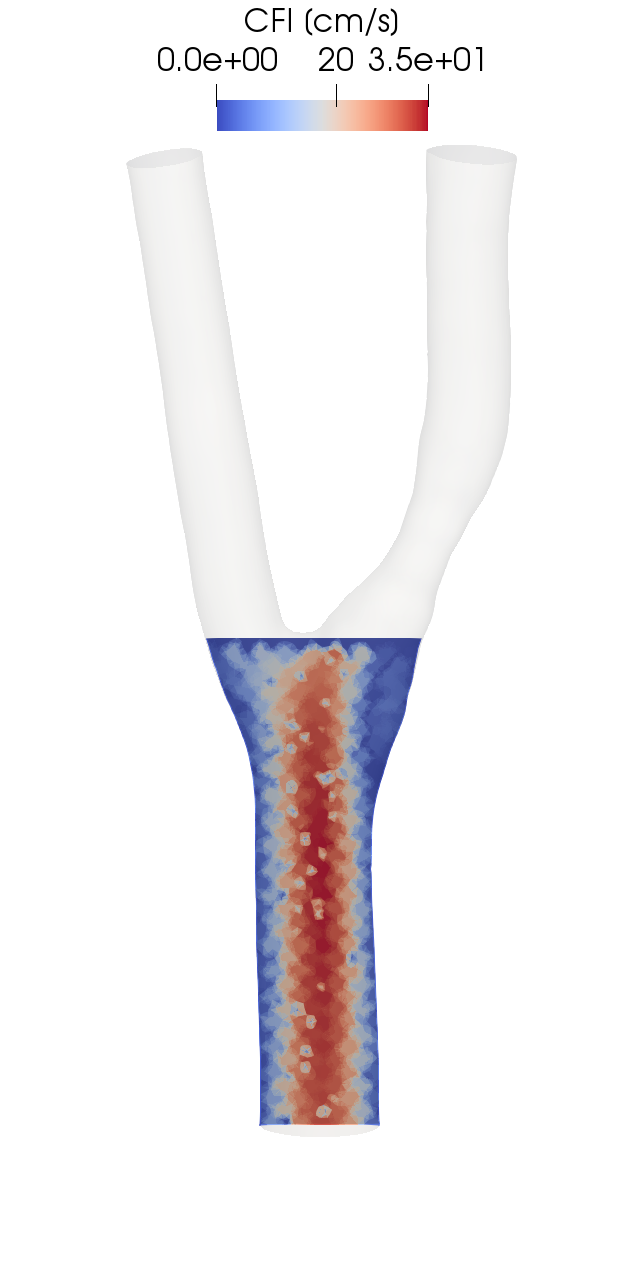} \label{fig:sint_blockage_CFI}}
    \subfigure[$A_{233,40}^{\text{pbdw}}(\omega)$]{\includegraphics[height = 7 cm]{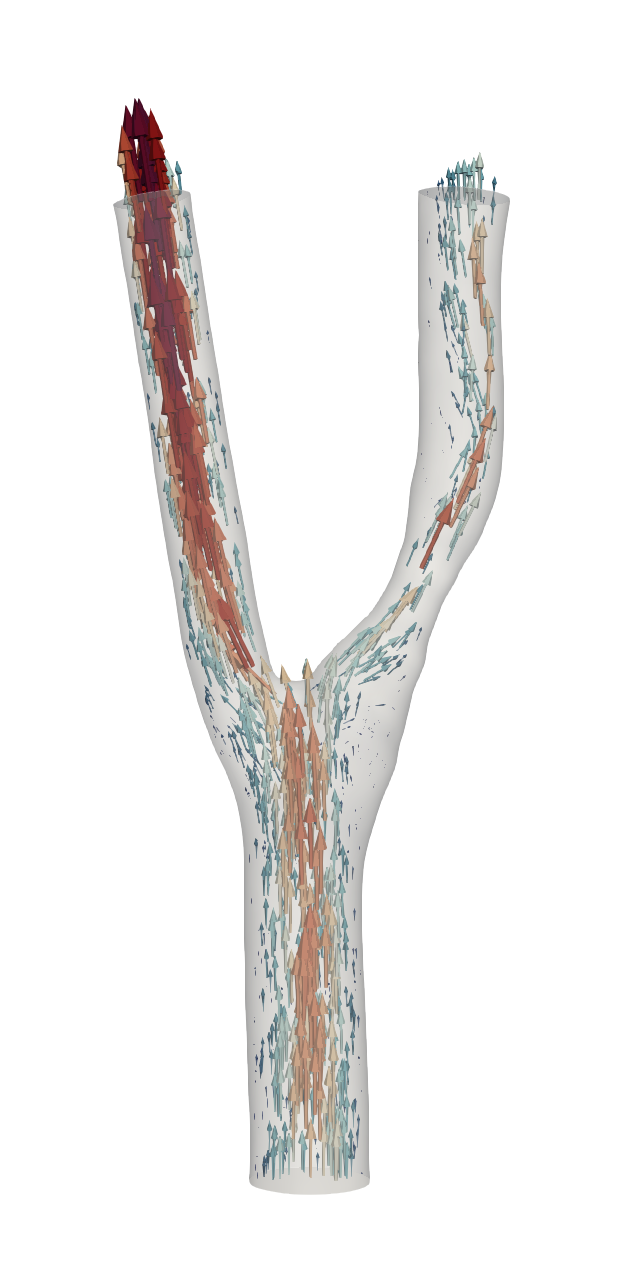}}
    \subfigure[$u - A_{233,40}^{\text{pbdw}}(\omega)$]{\includegraphics[height = 7 cm]{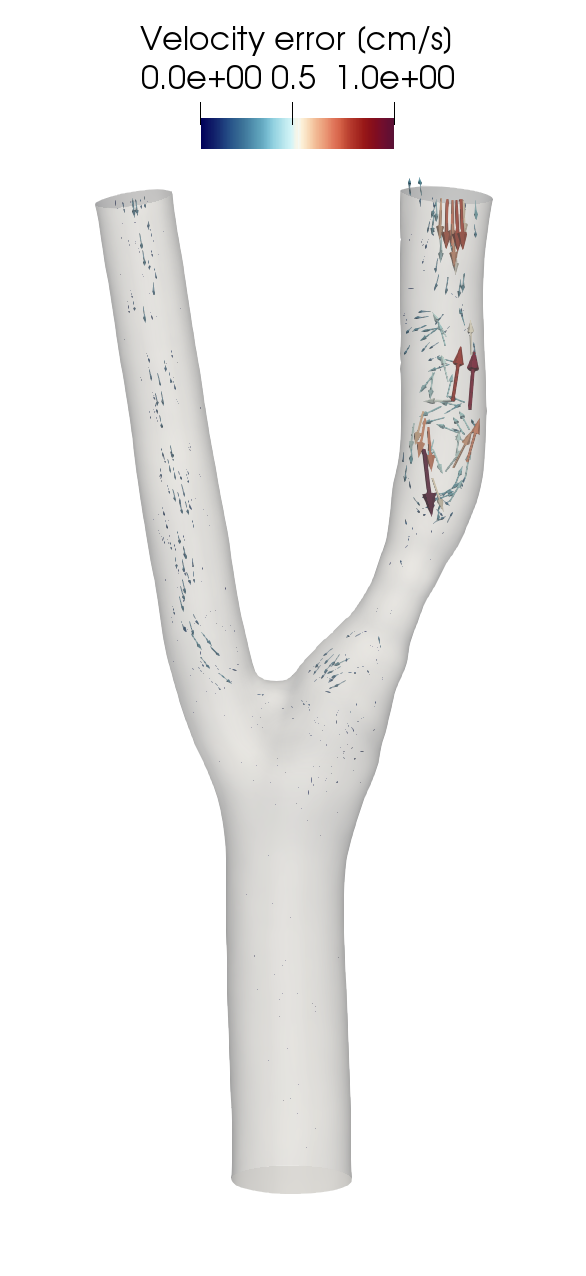}}
    \subfigure[$A_{233,40}^\dd(\omega)$]{\includegraphics[height = 7 cm]{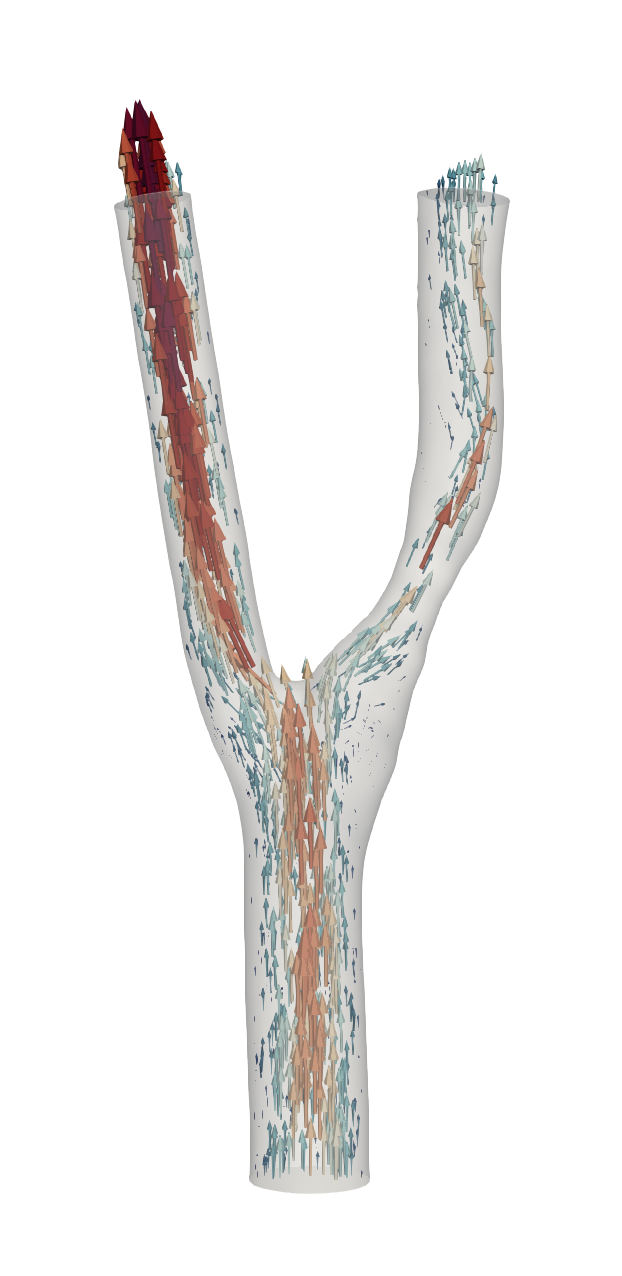}}
    \subfigure[$u - A_{233,40}^\dd(\omega)$]{\includegraphics[height = 7 cm]{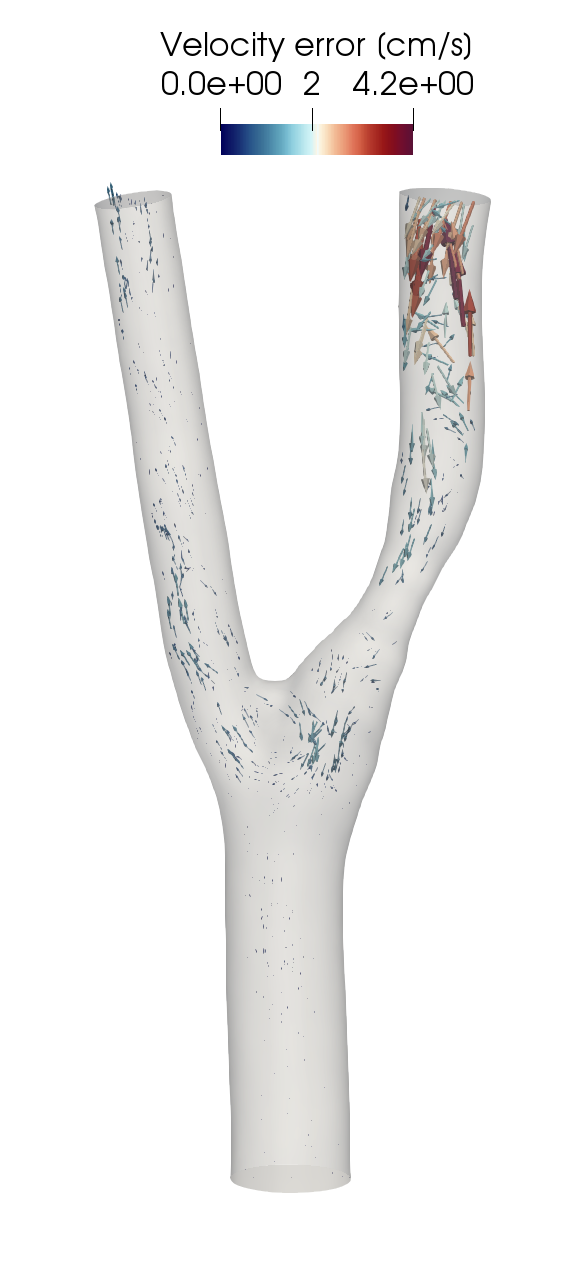}}\\
    \subfigure{\includegraphics[height = 1 cm]{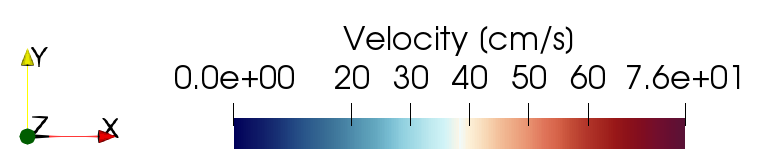}}
    \caption{\color{felipe}{Reconstruction example of a CFI partial observation at peak systole for blockage case using 40 modes. No flow information in the carotid branches is observed. Nevertheless, the algorithms are capable of reconstructing the velocity field in the whole computational domain.}}
    \label{fig:blockage-image-CFI}
\end{figure}

\begin{figure}
	\centering
    \subfigure[$\omega$]{\includegraphics[height = 7 cm]{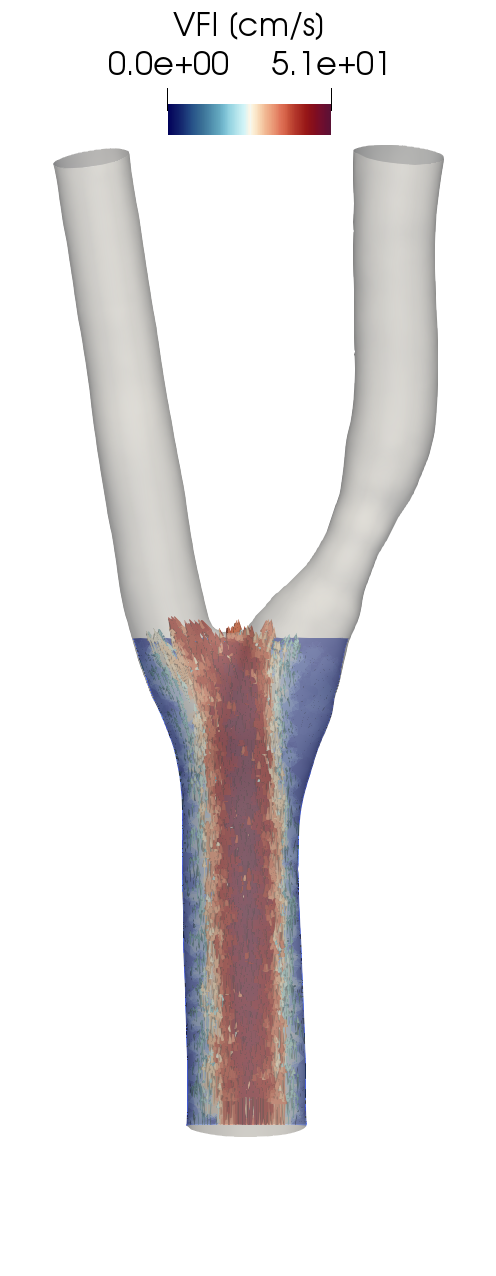} \label{fig:sint_blockage_CFI}} \hspace{.5cm}
    \subfigure[$A_{233,40}^{\text{pbdw}}(\omega)$]{\includegraphics[height = 7 cm]{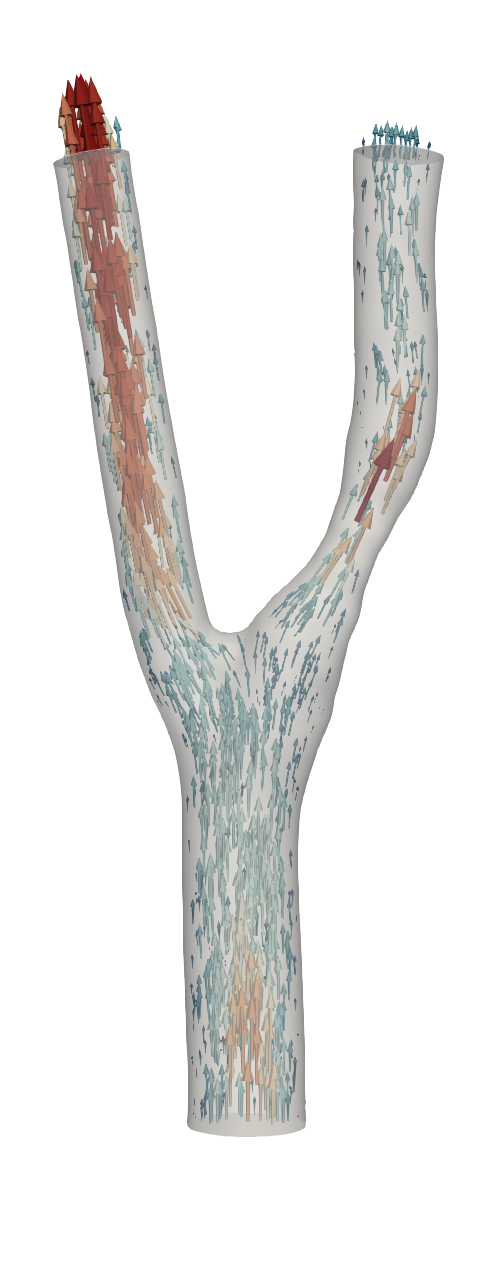}}\hspace{.5cm}
    \subfigure[$u - A_{233,40}^{\text{pbdw}}(\omega)$]{\includegraphics[height = 7 cm]{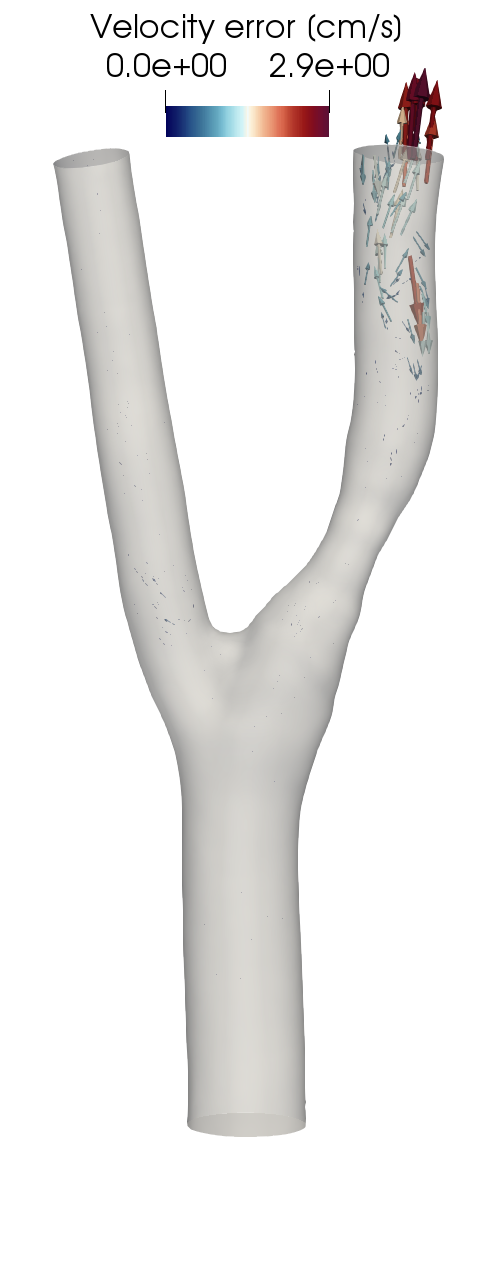}}\hspace{.5cm}
    \subfigure[$A_{233,40}^\dd(\omega)$]{\includegraphics[height = 7 cm]{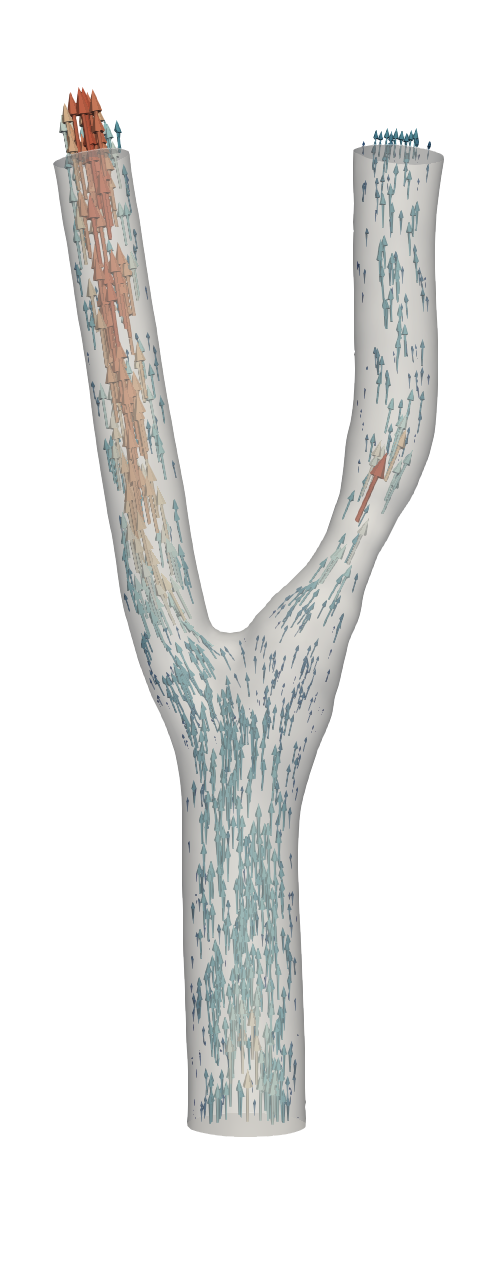}}\hspace{.5cm}
    \subfigure[$u - A_{233,40}^\dd(\omega)$]{\includegraphics[height = 7 cm]{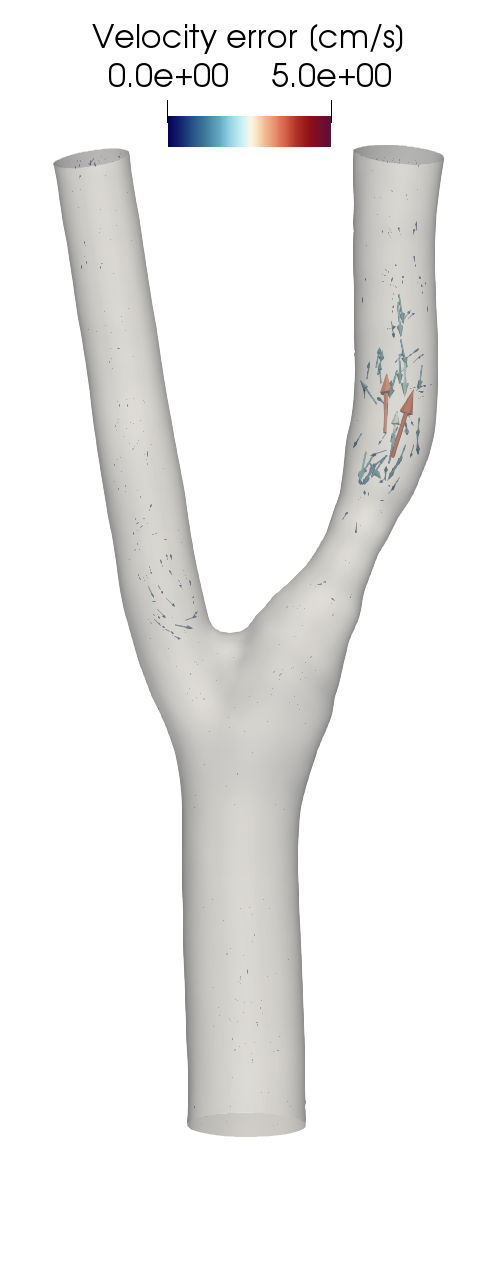}} \\
    \subfigure{\includegraphics[height = 1 cm]{./figs/field_CFI/colorbar.png}}
    \caption{\textcolor{felipe}{
Reconstruction example of a VFI partial observation with same information of that of Figure \ref{fig:blockage-image-CFI}. The method provides a slightly better reconstruction quality than that of the CFI image.}}
    \label{fig:blockage-image-VFI}
\end{figure}

\section{Conclusions and perspectives}
We have investigated the construction of state estimation techniques with different reduced models. We have compared their accuracy for the problem of reconstructing the blood  flow velocity field for which partial information is given by Doppler ultrasound images. The numerical tests were synthetically generated and mimic a real context in medical applications. The results show that the classical linear PBDW method with the classical POD reduced model can be outperformed by other models that are built more specifically for the reconstruction task. In particular, the nonlinear reconstruction method built by manifold partitioning and a local POD basis on each partition outperforms all the other bases choices. It presents a good trade-off between simplicity and efficiency and its accuracy is in some cases even ten times better than the classical PBDW. The data-driven approach gave less performant results than the one with local POD bases but superior to the classical approach. We think that the specific application may have had an important impact on the observed performance. Since we build a space $V_n(\omega)$ from the Doppler measurements $\omega$, they may not deliver enough information to learn a good reduced model, especially in the diastole phase where $\omega$ is close to zero. As a result, it seems important to investigate further the performance of this method in other reconstruction problems in order to gain better knowledge on its range of applicability.

\section*{Acknowledgments}
This research was partially supported by the Emergence Project  ``Models and Measurements'' of the Paris City Council. Also, authors thankfully acknowledge the financial support of the ANID Ph. D. Scholarship 72180473.

\section*{References}
\bibliographystyle{unsrt}
\bibliography{literature}

\end{document}